\documentclass[12pt]{elsarticle}

\usepackage{lineno,hyperref}
 \usepackage{float}
\usepackage{amsmath,amssymb,amsfonts}
\usepackage{xcolor,graphicx,caption}
\usepackage{amsthm} 
\usepackage[shortlabels]{enumitem}
\usepackage[scale=0.8]{geometry} 
\usepackage{comment}

\usepackage[utf8]{inputenc}
\usepackage{bm}
\usepackage{mathrsfs}
\usepackage{fullpage} 

\usepackage{mathtools}
\usepackage{titlesec}
\titleformat{\section}[block]{\bfseries\filcenter}{}{1em}{}

\allowdisplaybreaks
\modulolinenumbers[5]

\newtheorem{The}{Theorem}[section]
\newtheorem{Lem}[The]{Lemma}
\newtheorem{Cor}[The]{Corollary}
\newtheorem{Rem}[The]{Remark} 
\newtheorem{Def}[The]{Definition}
\newtheorem{Pro}[The]{Proposition}

\newcommand{\R}{{\mathbb R}}

\newcommand{\schwartz}{\mathscr{S}}

\newcommand{\tempered}{\mathscr{S}^{\prime}}

\newcommand{\fraclaplace}{(-\Delta)^s}
\newcommand{\fourier}{\mathcal{F}}
\newcommand{\ifourier}{\mathcal{F}^{-1}}

\newcommand{\dimens}{n}

\newcommand{\abs}[1]{\left\lvert #1 \right\rvert}
\newcommand{\aabs}[1]{\left\lVert #1 \right\rVert}
\DeclareMathOperator{\spt}{spt}

\usepackage{natbib}
\begin{document}

\begin{center}
{\Huge Uniqueness for the fractional Calder\'on problem with quasilocal perturbations}
\vspace{3mm}

\large{Giovanni Covi}
\vspace{10mm}

\textbf{Abstract}
\end{center}

We study the fractional Schr\"odinger equation with quasilocal perturbations and show that the qualitative unique continuation and Runge approximation properties hold in the assumption of sufficient decay. Quantitative versions of both results are also obtained via a propagation of smallness analysis for the Caffarelli-Silvestre extension. The results are then used to show uniqueness in the inverse problem of retrieving a quasilocal perturbation from DN data under suitable geometric assumptions. Our work generalizes recent results regarding the locally perturbed fractional Calder\'on problem.

\hrulefill

\section{Introduction}\label{sec:intro}

Let $s\in(0,1)$, assume that $\Omega\subset\R^\dimens, n\geq 1$ is a bounded open set, and let  $\Omega_e:=\R^\dimens\setminus\overline{\Omega}$ be its exterior. Given any set $A\subset\mathbb R^n$ and $r\geq 0$, define $N(A,r):= \{   x\in\mathbb R^n : \mbox{dist}(x,A)\leq r\}$ to be the neighbourhood of $A$ of radius $r$. We study an inverse problem for the fractional Schr\"odinger equation with a quasilocal perturbation
\begin{equation}
\label{main-eq}
\begin{array}{rrl}
\fraclaplace u+\Psi u & =0 & \text{in} \ \Omega \\
u & =f & \ \text{in} \ \Omega_e
\end{array}.
\end{equation}
In formula \eqref{main-eq}, $\fraclaplace$ is the pseudo-differential operator ($\Psi$DO) defined by $\fraclaplace u:=\ifourier(\abs{\cdot}^{2s}\hat{u})$, while the perturbation $\Psi$ is \emph{quasilocal} in the sense that there exists a \emph{decay function} $\mu: \mathbb R^+ \rightarrow \mathbb R^+_0$ with $\lim_{r\rightarrow\infty}\mu(r) = 0$ such that for all $u\in H^s$ and $r>0$ one has $$ \|\Psi u\|_{H^{-s}(N(\spt u,r)_e)} \leq \mu(r)\|u\|_{H^s} $$
(see Section \ref{sec:pre} for the exact definitions). Observe that both $(-\Delta)^s$ and $\Psi$ are nonlocal operators in general. In order to ensure the well-posedness of the direct problem \eqref{main-eq}, we always assume that $0$ is not a Dirichlet eigenvalue of $\fraclaplace+\Psi$, that is
\begin{equation*}
\text{If} \ u\in H^s(\R^\dimens) \ \text{solves} \ \fraclaplace u+\Psi u=0 \ \text{in} \ \Omega \ \text{and} \ u|_{\Omega_e}=0, \ \text{then} \ u\equiv 0.
\end{equation*} The goal of the inverse problem is to recover the perturbation $\Psi$ from the relative Dirichlet-to-Neumann map. This is the map $\Lambda_\Psi\colon H^s(\Omega_e)\rightarrow (H^s(\Omega_e))^*$ which associates each exterior datum $f$ to the corresponding nonlocal Neumann boundary value (see Section \ref{sec:wellposedness-DN}). In this article we are particularly concerned with the problem of uniqueness: given two quasilocal perturbations $\Psi_1, \Psi_2$, we wonder whether the fact that their DN maps $\Lambda_{\Psi_1}, \Lambda_{\Psi_2}$ coincide is enough to conclude that $\Psi_1 = \Psi_2$. 
\vspace{3mm}

We start our analysis from an easier case, that of \emph{perturbations of finite propagation}. These are special quasilocal perturbations $\Psi$ for which there exists a real number $R\geq 0$, the so called \emph{propagation of $\Psi$}, such that $$\mbox{spt}(\Psi u) \subseteq N(\spt u,R) \quad\mbox{for all}\quad u\in H^s .$$

For this case, we are able to prove the following theorem, showing that uniqueness does hold in the inverse problem if a geometric assumption is made: 

\begin{The}\label{main-finite}
Let $\Omega\subset\mathbb R^n$ be a bounded open set, $s\in(0,1)$, and let $\Psi_1, \Psi_2 \in B_0(H^s,H^{-s})$ have finite propagation. Assume that there exist two open sets $W_1, W_2\subset\Omega_e$ far enough from $\Omega$ and among themselves such that $$ \Lambda_{\Psi_1}f|_{W_2} = \Lambda_{\Psi_2}f|_{W_2} $$
for all $f\in C^\infty_c(W_1)$. Then $\Psi_1 = \Psi_2$ as operators in $B(\widetilde H^s(\Omega),H^{-s}(\Omega))$.
\end{The}

The proof of Theorem \ref{main-finite} is based on two main ingredients. The first one of them is an Alessandrini identity, which we prove in Lemma \ref{alex} and whose use is relating the difference of the DN maps to the difference of the perturbations. The second ingredient is a Runge approximation property, showing that any $\widetilde H^s(\Omega)$ function can be approximated by solutions of \eqref{main-eq} whose exterior data are supported far enough from $\Omega$. In order to prove this, we shall make use of a UCP result with a similar geometric condition.
\vspace{3mm}

Given the nonlocality of the operators involved, which in some sense is "worse" for general quasilocal perturbations than for perturbations of finite propagation, the technique described above can not simply be applied to the main case of quasilocal perturbations as it is. We start by proving a qualitative UCP result for quasilocal perturbations, which we believe is interesting on its own right and thus state as a theorem:

\begin{The}[UCP for $(-\Delta)^s + \Psi$, $\Psi$ quasilocal]\label{UCP_for_quasilocal_perturbation}\label{th:quasilocal} Let $W\subset\mathbb R^n$ be an unbounded open set and $\Omega\subset W_e$ be a bounded, open and smooth set. Assume that $u\in \widetilde H^s(\Omega)$, that $\|(-\Delta)^su\|_{L^2(\Omega)}<\infty$, and $((-\Delta)^su + \Psi u)|_W=0$, with $\Psi\in B(H^s,H^{-s})$ a quasilocal perturbation whose decay function $\mu$ satisfies 
\begin{equation}\label{mu-condition}
   \lim_{r\rightarrow\infty} c_M(r) |\log \mu(r)|^{-\sigma_M(r)}=0\;. 
\end{equation}
for some $M\geq 0$. Then $u\equiv 0$.\end{The}

Here the functions $c_M(r), \sigma_M(r)$ are the same ones appearing in the fractional Calder\'on stability result from \cite{RS17} (check also our Sections \ref{subsec:log} and \ref{sec:UCP-Runge} for the exact definitions). Using Theorem \ref{th:quasilocal}, we are able to prove a \emph{qualitative} Runge approximation property. However, since the nonlocality of the operator appearing in the Alessandrini identity calls for an estimate of the exterior datum corresponding to the solution used to test said identity, we rather need a \emph{quantitative} Runge approximation result. We are able to obtain this property by means of a quantitative UCP and a technique similar to the one from \cite{RS17}, but suited for quasilocal perturbations. Eventually, we get the following uniqueness result for the inverse problem:

\begin{The}\label{main-quasi}
Let $\Omega\subset\mathbb R^n$ be a bounded, open and smooth set, $s\in(0,1)$, and let $W_1, W_2 \subset \Omega_e$ be open, unbounded and smooth sets such that $ |x_1 - x_2| \geq \max_{j=1,2}\{ \mbox{dist}(x_j,\Omega) \} $ for all $x_1\in W_1, x_2\in W_2$. Let $\Psi_1, \Psi_2 \in B(L^2,L^2)$ be quasilocal perturbations with decay function $\mu$ satisfying condition \eqref{mu-condition} for $M:=\max\{\|\Psi_1\|,\|\Psi_2\|\}$ and both $W_1, W_2$. Finally, assume that
$$ \Lambda_{\Psi_1}f|_{W_2} = \Lambda_{\Psi_2}f|_{W_2} $$
for all $f\in C^\infty_c(W_1)$. Then $\Psi_1=\Psi_2$ as operators in $B(\widetilde H^s(\Omega),H^{-s}(\Omega))$.
\end{The}

In all three theorems we are making some geometric assumptions regarding the unboundedness of $W$ and the relative distances between the measuring sets $W_1, W_2$ and the studied domain $\Omega$. We believe that the geometric assumptions appearing in Theorems \ref{main-finite} and \ref{th:quasilocal} are essential for the proofs. In contrast, the geometric assumptions of Theorem \ref{main-quasi} serve the purpose of simplifying the proof and can probably be refined (see Remark \ref{geom-remark}).   

\subsection{Connection to the earlier literature}

In the technique known as electrical impedance tomography, one attempts to recover information regarding the electrical properties of an object by means of measurements of voltage and current performed on its surface. This has lead to the classical Calder\'on problem, which consists in recovering an unknown potential in the interior of a bounded domain $\Omega$ from a Dirichelet-to-Neumann (DN) map representing measurements on $\partial\Omega$~\cite{Uh09, Uh14}. 

In the seminar paper~\cite{GSU20} a nonlocal counterpart to the classical Calder\'on problem was first introduced: this consists in recovering a potential $q$ associated to a fractional Schr\"odinger operator $(-\Delta)^s+q$ from a DN map encoding measurements in the exterior $\Omega_e$. The cited work proved uniqueness in the case of potentials $q\in L^\infty(\Omega)$ for exponents $s\in(0,1)$. A subsequent paper~\cite{RS17} extended the result to rough potentials belonging to certain spaces of Sobolev multipliers. The proof of uniqueness was further generalized in~\cite{CMR20} to include all positive fractional exponents $s\in\mathbb R^+\setminus\mathbb Z$ in the case of rough potentials.

In all the cited works, the potential $q$ can be interpreted as a local $0$-th order perturbation of the original fractional Laplacian. The paper~\cite{CLR18} introduced first order perturbations of the form $b\cdot\nabla + q$ in the assumption that $s\in (1/2, 1)$, and showed uniqueness in such case. The result was eventually extended to higher order rough perturbations and all exponents $s\in\mathbb R^+\setminus\mathbb Z$ in \cite{CMRU20}; as observed by the authors of such article, all local perturbations of the fractional Laplacian are necessarily of this kind. Some specific cases of nonlocal perturbations were studied in \cite{BGU21, LO21}.

The problem of stability for the fractional Calder\'on problem has also been addressed, and exponential instability was shown in~\cite{RS17, RS17d}. More general instability mechanisms for inverse problems have been the focus of \cite{KRS21}. Moreover, the fractional Calder\'on problem has been solved even for a single measurement in ~\cite{GRSU20}. Other settings include the fractional magnetic Schr\"odinger equation~\cite{C20a, CMR20, Li20a, Li20b, Li21}, the fractional conductivity equation~\cite{C20}, the fractional heat equation~\cite{LLR19,RS17a} and a fractional elasticity equation with constant coefficients \cite{LL21}. A semilinear fractional Schr\"odinger equation was studied in~\cite{LL19, Li20a, Li20b, Li21}. Many more details about fractional Calder\'on problems can be found in the survey~\cite{S17}. 
\vspace{3mm}

By studying a fractional Schr\"odinger equation with a special family of nonlocal perturbations, the present work can be considered a step towards the study of more general nonlocal perturbations of the fractional Laplacian. We expect that our results serve the purpose of highlighting the possible obstacles and assumptions which may be needed by a more general theory of nonlocally perturbed Calder\'on-type problems. We also believe that quasilocal perturbations may be interesting by themselves, given that some quite natural examples of nonlocal perturbations, such as (weighted) convolutions against Schwartz functions, belong to this family (see Example 4 in Section \ref{sec:examples}).      
Finally, it may be argued that our Theorem \ref{th:quasilocal} bears a connection to the fractional Landis conjecture studied in \cite{RW19} (see Remark \ref{landis}).
\vspace{3mm}

Besides the purely mathematical interest, the study of nonlocal operators, and in particular of the fractional Laplacian, is motivated by the numerous applications that have been found for them in the natural sciences (see e.g.~\cite{BV16, MK00, RS15} and the relative references). Whenever a diffusion process presents a relation between time and mean displacement other than quadratic, it is considered anomalous. The fractional Laplacian and related operators appear in the study of such nonlocal diffusion phenomena, and therefore arise in models for mathematical finance \cite{AB88, Le04, Sc03}, engineering \cite{GO08}, turbulent fluid dynamics \cite{Co06, DG13}, physics \cite{DGLZ12, Er02, GL97, MK00, ZD10}, fractional quantum mechanics~\cite{La18, La02}, and ecology \cite{Hum10, MV18, RR09}, just to name a few. 

\subsection{Organization of the article}
The rest of the article is organized as follows. Section \ref{sec:pre} contains preliminaries about the fractional Sobolev spaces involved in the discussion, the fractional Laplacian, quasilocal operators and logarithmic estimates for the perturbed fractional Schr\"odinger equation. In Section \ref{sec:wellposedness-DN} we study the well-posedness of the direct problem and introduce the DN map. The qualitative unique continuation and Runge approximation properties are discussed in Section \ref{sec:UCP-Runge}. Section \ref{sec:quantiRunge} is devoted to proving the quantitative Runge approximation property, which is fundamental in the proof of the main theorems in Section \ref{sec:proofs-main}. Section \ref{sec:examples} contains many non-trivial examples of the previously introduced operators, adding clarity to the discussion. Finally, the Appendix deals with the proof of the quantitative unique continuation property, in particular showing how the constants involved in the logarithmic stability estimates for the fractional Schr\"odinger equation depend on the geometry and parameters of the problem.

\subsection{Acknowledgements} The author was supported by an Alexander-von-Humboldt postdoctoral fellowship; he also wishes to thank professors Angkana R\"uland and Mikko Salo for helpful discussion.

\section{Preliminaries}\label{sec:pre}

\subsection{Sobolev spaces and the fractional Laplacian}
The fractional Sobolev space of order $r\in\R$ based on $L^2$ is defined as
\begin{equation*}
H^{r}(\R^\dimens):=\{u\in\tempered(\R^\dimens): \ifourier(\langle\cdot\rangle^r\hat{u})\in L^2(\R^\dimens)\},
\end{equation*}
where $\mathscr S(\mathbb R^n), \mathscr S'(\mathbb R^n)$ respectively indicate the sets of Schwartz functions and tempered distributions, $\hat{u}=\fourier (u)$ is the Fourier transform, $\mathcal F^{-1}(u)$ is the inverse Fourier transform, and $\langle x\rangle:=(1+\abs{x}^2)^{1/2}$. The space $H^{r}(\R^\dimens)$ can be equipped with the norm 
\begin{equation*}
\aabs{u}_{H^{r}(\R^\dimens)}:=\aabs{\ifourier(\langle\cdot\rangle^r\hat{u})}_{L^2(\R^\dimens)}.
\end{equation*}

For $\Omega, F\subset\R^\dimens$ an open and a closed set we define the following Sobolev spaces:
\begin{align*}
H_F^r(\R^\dimens) &=\{u\in H^r(\R^\dimens): \spt(u)\subset F\} \\
\widetilde{H}^r(\Omega)&= \ \text{closure of} \ C_c^{\infty}(\Omega) \ \text{in} \  H^r(\R^\dimens) \\
H^r(\Omega)&=\{u|_\Omega: u\in H^r(\R^\dimens)\} \\
H_0^r(\Omega)&= \ \text{closure of} \ C_c^{\infty}(\Omega) \ \text{in} \ H^r(\Omega). 
\end{align*}
One can prove that the inclusions $\widetilde{H}^r(\Omega)\subset H_0^r(\Omega)$ and $\widetilde{H}^r(\Omega)\subset H^r_{\overline{\Omega}}(\R^\dimens)$ always hold. If~$\Omega$ is Lipschitz, then we also have $\widetilde{H}^r(\Omega)=H_{\overline{\Omega}}^r(\R^\dimens)$ for all $r\in\R$ and $H_0^r(\Omega)=H^r_{\overline{\Omega}}(\R^\dimens)$ for all $r\geq 0$ such that $r-\frac{1}{2} \notin \mathbb Z$ (see e.g.~\cite{McLean}, Theorems 3.29, 3.33). Finally, by \cite[Theorem 3.3]{CWHM-sobolev-spaces-on-non-lipchtiz-domains} the following identifications with the dual spaces hold for all $r\in\R$: $(\widetilde{H}^r(\Omega))^*=H^{-r}(\Omega)$ and $(H^r(\Omega))^*=\widetilde{H}^{-r}(\Omega)$.

Since we are considering nonlocal operators, we impose exterior values for the equation. Therefore, we define the abstract trace space $X:=H^r(\R^\dimens)/\widetilde{H}^r(\Omega)$: two functions $f_1, f_2\in H^r(\R^\dimens)$ belong to the same class in $X$ if and only if they agree in $\Omega_e$, in the sense that $f_1-f_2 \in \widetilde H^r(\Omega)$. If~$\Omega$ is Lipschitz, then $X=H^r(\Omega_e)$, as proved in ~\cite[p.463]{GSU20}.

The fractional Laplacian of order $s\in(0,1)$ can be defined in many different ways (see e.g.~\cite{DL21, Kw15, LPGS}). Here we use the symbol definition: for us the fractional Laplacian is given by $$\fraclaplace \varphi:=\ifourier(\abs{\cdot}^{2s}\hat{\varphi})$$ for $\varphi\in\schwartz(\R^\dimens)$. Given the density of $\schwartz(\R^\dimens)$ in $H
^r(\R^\dimens)$, the fractional Laplacian can be extended to act as a bounded operator $\fraclaplace\colon H^r(\R^\dimens)\rightarrow H^{r-2s}(\R^\dimens)$ for all $r\in\R$. This nonlocal operator has two important properties which we use in our proofs: the unique continuation property (UCP) and the Poincar\'e inequality.
\begin{Lem}[UCP]\label{ucpfraclap}
Let $s\in(0,1)$, $r\in\R$ and $u\in H^{r}(\R^\dimens)$. If $(-\Delta)^s u|_V=0$ and $u|_V=0$ for $V\subset\R^\dimens$ open and non-empty, then $u\equiv0$.
\end{Lem}

\begin{Lem}[Poincar\'e inequality]
Assume $s\in(0,1)$ and $u\in H^s_K(\R^\dimens)$, where $K\subset\R^\dimens$ is a compact set. There exists a constant $c=c(n, s, K)> 0$ such that
\begin{equation*}
\aabs{u}_{L^2(\R^\dimens)}\leq c\aabs{(-\Delta)^{s/2}u}_{L^2(\R^\dimens)}.
\end{equation*}
\end{Lem}
 The first one is needed in the qualitative Runge approximation result, while the second one is used for proving the well-posedness of the direct problem. The proofs of both Lemmas can be found in \cite{GSU20}. It is worth noticing that both the definition of the fractional Laplacian and the listed results can be generalized to the case $s\in \mathbb R^+\setminus\mathbb Z$, as it was proved in \cite{CMR20}.

\subsection{Quasilocal operators}
We define $B(X,Y)$ to be the space of bounded linear operators between the normed spaces $X$ and $Y$. In particular, this means that if $s\in (0,1)$ and $\Psi \in B(H^s,H^{-s})$, then there exists a constant $C>0$ such that $\|\Psi f\|_{H^{-s}}\leq C \|f\|_{H^s} $ for all $f\in H^s$. Once equipped with the operator norm
$$ \|\Psi\| := \|\Psi\|_{B(H^s,H^{-s})} := \sup \{ \|\Psi f\|_{H^{-s}} : \|f\|_{H^s}=1 \}\;, $$
$B(H^s,H^{-s})$ is a normed space. Given that $$\widetilde H^s(\Omega)\subset H^s \subset L^2 \subset H^{-s} \subset H^{-s}(\Omega)$$ holds for all $s>0$, one has the following inclusions: $$ B(L^2,L^2)\subset B(H^s,H^{-s})\subset B(\widetilde H^s(\Omega),H^{-s}(\Omega))\;. $$
Let $B_0(H^s,H^{-s})$ be the closure of $B(L^2,L^2)$ in the norm of $B(H^s,H^{-s})$. This means that for all $\Psi \in B_0(H^s,H^{-s})$ and $\epsilon >0$ we can find $\Psi_1\in B(L^2,L^2)$ and $\Psi_2\in B(H^s,H^{-s})$ such that
$$ \Psi = \Psi_1 + \Psi_2\;, \qquad \|\Psi_2\|_{B(H^s,H^{-s})}<\epsilon\;. $$

The spaces defined above are related to the multiplier spaces $M(H^s\rightarrow H^{-s})$ studied in \cite{CMRU20}: in fact, one shows that $M(H^s\rightarrow H^{-s})\subset B(H^s,H^{-s})$ holds, and thus also $M_0(H^s\rightarrow H^{-s})\subset B_0(H^s,H^{-s})$.

\begin{Def}
The operator $\Psi\in B(H^s,H^{-s})$ is said to have \emph{finite propagation} if there exists $R\geq 0$ such that spt$(\Psi u) \subseteq N(\spt u,R)$ for all $u\in H^s$. The infimum of all such $R$ is called the \emph{propagation of $\Psi$}, and it is indicated by $p(\Psi)$.
\end{Def}

\begin{Def}
The operator $\Psi\in B(H^s,H^{-s})$ is said to be \emph{quasilocal} if there exists a \emph{decay function} $\mu: \mathbb R^+ \rightarrow \mathbb R^+_0$ with $\lim_{r\rightarrow\infty}\mu(r) = 0$ such that for all $u\in H^s$ and $r>0$ one has $ \|\Psi u\|_{H^{-s}(N(\spt u,r)_e)} \leq \mu(r)\|u\|_{H^s} $.
\end{Def}

In the next lemma we prove some properties which hold true for quasilocal operators.

\begin{Lem}(Properties of quasilocal operators)
\label{quasiprops}
\begin{enumerate}[label=\arabic*.]
    \item If $P\in B(H^s,H^{-s})$ is local, then it has finite propagation. If $\Psi\in B(H^s,H^{-s})$ has finite propagation, then it is quasilocal.
    \item Let $\Psi \in B(L^2,L^2)$ be quasilocal with decay function $\mu$. When interpreted as element of $B(H^s,H^{-s})$, $\Psi$ is again quasilocal and has the same decay function $\mu$. 
    \item Let $\Psi\in B(H^s,H^{-s})$ be quasilocal with decay function $\mu$. The adjoint operator $\Psi^*\in B(H^s,H^{-s})$ is again quasilocal and has the same decay function $\mu$.
    \item If $\Psi\in B_0(H^s,H^{-s})$, then $\Psi^*\in B_0(H^s,H^{-s})$ as well.
\end{enumerate}
\end{Lem}

\begin{proof}
\begin{enumerate}[label=\arabic*.]
    \item If $P$ is local then spt$(P u) = \spt u = N(\spt u,0)$ for all $u\in H^s$, and thus $P$ has finite propagation $p(P)=0$. If $\Psi$ has finite propagation, then the function $\mu(r):= \chi_{\{r\leq p(\Psi)\}}(r)\|\Psi\|$ has the properties required by quasilocality. However, not all quasilocal operators have finite propagation, and not all operators of finite propagation are local. This is shown by the examples in Section \ref{sec:examples}.
    \item This property is proved by the following computation: $$  \|\Psi u\|_{H^{-s}(N(\spt u, r)_e)} \leq \|\Psi u\|_{L^2(N(\spt u, r)_e)} \leq \mu(r)\|u\|_{L^2} \leq \mu(r)\|u\|_{H^s}\;. $$
    \item The adjoint $\Psi^*$ is the operator in $B(H^s,H^{-s})$ defined by $\langle \Psi^* u,v \rangle := \langle u, \Psi v \rangle$ for all $u,v\in H^s$. Given $u\in H^s$ and $r>0$, we want to prove that $$\|\Psi^* u\|_{H^{-s}(N(\spt u, r)_e)}\leq \mu(r)\|u\|_{H^s}\;.$$ Observe that by definition $H^{-s}(N(\spt u, r)_e) = \widetilde H^s(N(\spt u,r)_e)^*$. Moreover, for $u,v\in H^s$ we have $ u\in \widetilde H^s(N(\spt v,r)_e) \Leftrightarrow v\in\widetilde H^s(N(\spt u,r)_e) $. With these observations in mind, we compute for $u\in H^s$ and $v\in\widetilde H^s(N(\spt u,r)_e)$
    \begin{align*}
     |\langle \Psi^*u,v \rangle| = |\langle u,\Psi v \rangle| \leq \|u\|_{H^s}\|\Psi v\|_{H^{-s}(N(\spt v,r)_e)} \leq \mu(r) \|u\|_{H^s}\|v\|_{H^s}\;,
    \end{align*}
    and thus
    \begin{align*}
  \|\Psi^* u\|_{H^{-s}(N(\spt u, r)_e)} = \sup \{ |\langle \Psi^*u,v \rangle| : v\in\widetilde H^s(N(\spt u,r)_e), \|v\|_{H^s}=1 \} \leq \mu(r) \|u\|_{H^s}\;.
    \end{align*}
    \item Since $\Psi^*$ is known to be a bounded operator between $H^s$ and $H^{-s}$, we only need to show the approximation property. Let $\{\psi_n\}_{n\in\mathbb N}\subset B(L^2,L^2)$ be a sequence such that $\|\Psi-\psi_n\|_{B(H^s,H^{-s})}<1/n$. Then
    \begin{align*}
        \|\Psi^*-\psi_n^*\|_{B(H^s,H^{-s})} & = \sup \{ \|(\Psi^*-\psi_n^*) u\|_{H^{-s}} : \|u\|_{H^s}=1 \} \\ & = \sup \{ \sup \{ |\langle (\Psi^*-\psi_n^*)u,v \rangle| : \|v\|_{H^s}=1 \} : \|u\|_{H^s}=1 \} \\ & = \sup\{ |\langle u, (\Psi-\psi_n)v \rangle| : \|u\|_{H^s}=\|v\|_{H^s}=1 \} \\ & \leq \|\Psi-\psi_n\|_{B(H^s,H^{-s})}<1/n\;,
    \end{align*}
    and thus the sequence $\{\psi^*_n\}_{n\in\mathbb N}\subset B(L^2,L^2)$ approximates $\Psi^*$ in $B(H^s,H^{-s})$.
    \end{enumerate}
\end{proof}

\subsection{Logarithmic estimates for perturbed fractional Schr\"odinger equations}\label{subsec:log}

In the next theorem we give a logarithmic estimate for a fractional Schr\"odinger equation with a bounded perturbation on $L^2$:

\begin{The}\label{logestimate}
Let $\Omega\subset\mathbb R^n$, $n\geq 1$, be a bounded smooth domain, let $s\in(0,1)$, and let $W\subset\Omega_e$ be open. Let also $\Psi\in B(L^2(\Omega),L^2(\Omega))$ satisfy condition \eqref{eigenvalues} and $\|\Psi\|_{B(L^2(\Omega),L^2(\Omega))}\leq M$ for some constant $M>0$. There exist constants $c,\sigma$ depending on $\Omega, W, n, s$ and $M$ such that 
$$\|v\|_{H^{-s}(\Omega)}\leq c\, |\log(\|(-\Delta)^sw\|_{H^{-s}(W)})|^{-\sigma}\;,\qquad \mbox{ for all } v\in L^2(\Omega) : \|v\|_{L^2(\Omega)}=1\;,$$
where $w\in H^s(\mathbb R^n)$ solves $(-\Delta)^sw+\Psi w=v$ in $\Omega$ with $w|_{\Omega_e}=0$.
\end{The}

In the case of a local perturbation $\Psi$, Theorem \ref{logestimate} can be considered a quantitative version of the known UCP result for the locally perturbed fractional Laplacian. However, since we deal with nonlocal perturbations, in our case it is not possible to deduce the UCP for $(-\Delta)^s+\Psi$ directly from Theorem \ref{logestimate}. We shall instead provide an independent proof.
\vspace{2mm}

In the special case when $\Psi$ is a multiplication operator on $L^2(\Omega)$, Theorem \ref{logestimate} reduces to Theorem 1.3 in \cite{RS17}. Given that the proof of the cited theorem uses only the boundedness of $\Psi$, the proof for general $\Psi\in B(L^2(\Omega),L^2(\Omega))$ does not differ substantially from the one of the special case. However, for the sake of completeness, we sketch it in the Appendix.
\vspace{2mm}

It is also interesting (while not essential for our discussion) to highlight how the constants $c,\sigma$ depend on the sets and parameters of the Theorem, in particular $M$ and $W$. We postpone this discussion to the Appendix as well.

\vspace{2mm}

\section{Well-posedness \& DN maps}\label{sec:wellposedness-DN}
 
In this section we show the well-posedness of the direct problem
\begin{equation}
\label{directproblem}
\begin{split}
    (-\Delta)^s u+\Psi u & =F \mbox{ in } \Omega \\
    u & =f \mbox{ in } \Omega_e
\end{split}\;,
\end{equation}
where the perturbation $\Psi$ is assumed to be in $B(H^s,H^{-s})$. Observe in particular that in this section we are not making any quasilocality assumption. 
\\

In order to study the well-posedness of the direct problem \eqref{directproblem}, we also define the bilinear form associated to it, that is
$$ B_\Psi (v,w) := \langle (-\Delta)^{s/2} v, (-\Delta)^{s/2}w \rangle + \langle \Psi v,w \rangle \;.$$
The above definition makes sense for $v,w\in C^\infty_c(\mathbb R^n)$, but it can be extended to $v,w\in H^s$ by means of the following boundedness lemma:

\begin{Lem}[Boundedness lemma]\label{boundedness}
Let $\Psi\in B(H^s,H^{-s})$. The bilinear form $B_\Psi$ can be extended to act on $H^s\times H^s$.
\end{Lem}
\begin{proof}
The proof is just an easy computation:
\begin{align*} |B_\Psi(v,w)| & \leq |\langle (-\Delta)^{s} v, w \rangle| + |\langle \Psi v, w \rangle| \leq (\|(-\Delta)^sv\|_{H^{-s}} + \|\Psi v\|_{H^{-s}} )\|w\|_{H^s} \\ & \leq ( 1 + \|\Psi\|_{B(H^s,H^{-s})} )\|v\|_{H^{s}}\|w\|_{H^s} \leq C \|v\|_{H^{s}}\|w\|_{H^s}\;.
\end{align*}
\end{proof}

In the next coercivity lemma we prove another property of the bilinear form $B_\Psi$, which is fundamental in the proof of well-posedness.

\begin{Lem}[Coercivity lemma]\label{coercivity}
Let $\Psi\in B_0(H^s,H^{-s})$. The bilinear form $B_\Psi$ is coercive, that is there are constants $c_0, c_1 >0$ such that
$$ B_\Psi(v,v) \geq c_0\|v\|_{H^s}^2-c_1\|v\|_{L^2}^2 $$
for all $v\in \widetilde H^s(\Omega)$. 
\end{Lem}

\begin{proof}
We need to estimate the perturbation term $|\langle \Psi v, v \rangle|$ for $v\in H^s$. Since $\Psi\in B_0(H^s,H^{-s})$, by the definition of this space for every $\epsilon>0$ there exist $\Psi_1\in B(L^2,L^2)$ and $\Psi_2\in B(H^s,H^{-s})$ such that 
$$ \Psi = \Psi_1+\Psi_2\;, \qquad \|\Psi_2\|_{B(H^s,H^{-s})}< \epsilon\;. $$ Of course $c_1 := \|\Psi_1\|_{B(L^2,L^2)}> 0$ will also be a (possibly increasing) function of $\epsilon$. Thus
\begin{align*}
    |\langle \Psi v,v\rangle| \leq |\langle \Psi_1 v,v\rangle| + |\langle \Psi_2 v,v\rangle| \leq c_1\|v\|_{L^2}^2 + \epsilon\|v\|_{H^s}^2\;
\end{align*}
holds. Observe now that $\widetilde H^s(\Omega) \subset H^s_{\overline \Omega}(\mathbb R^n)$, and since $\Omega$ is bounded, $\overline\Omega$ is compact. This means that we can apply the fractional Poincar\'e inequality, which implies
\begin{align*}
    B_\Psi(v,v) & \geq \|(-\Delta)^{s/2} v\|_{L^2}^2 - |\langle \Psi v,v\rangle| \\ & \geq c_\Omega(\|(-\Delta)^{s/2} v\|_{L^2}^2 + \|v\|_{L^2}^2) - c_1\|v\|_{L^2}^2 - \epsilon\|v\|_{H^s}^2 \\ & \geq (c_\Omega-\epsilon)\|v\|_{H^s}^2 - c_1\|v\|_{L^2}^2 \;.
\end{align*}
Given that $\epsilon$ can be arbitrarily small, we get the wanted result by letting $c_0:=c_\Omega-\epsilon > 0$.
\end{proof}

With the two lemmas above, we can prove the well-posedness of the direct problem \eqref{directproblem}:

\begin{Pro}\label{wellposed}
Let $\Omega\subset\mathbb R^n$ be a bounded open set and $s\in(0,1)$. Assume $\Psi\in B_0(H^s,H^{-s})$, and let $c_1$ be the coercivity constant from Lemma \ref{coercivity}. There exists a 
countable set $\Sigma= \{\lambda_j\}_{j=1}^{\mathbb N}\subset (-c_1,\infty)$ with $\lambda_1 \leq \lambda_2 \leq ... \rightarrow \infty$ such that if $\lambda \in \mathbb R\setminus \Sigma$, then for any $f\in H^s$ and $F\in (\widetilde H^s(\Omega))^*$ there exists a unique $u\in H^s$ such that $u - f \in \widetilde H^s(\Omega)$ and $$B_\Psi (u,v)-\lambda \langle u, v\rangle_\Omega = F(v)\;,\qquad \mbox{for all } v\in \widetilde H^s(\Omega)\;.$$
One also has the estimate $$\|u\|_{H^s}\leq C(\|f\|_{H^s} + \|F\|_{(\widetilde H^s(\Omega))^*})\;.$$
\end{Pro}

\begin{proof}
The proof goes along the same lines as the analogous statements in \cite{GSU20} and \cite{CMRU20}. First of all, we homogenize the problem by assuming $\widetilde u := u-f$ and $\widetilde F := F-B_\Psi(f,\cdot) + \lambda \langle f,\cdot \rangle_\Omega$. Then, given the coercivity lemma, we know that $B_\Psi(\cdot,\cdot)+c_1\langle \cdot,\cdot\rangle$ defines an equivalent inner product on $\widetilde H^s(\Omega)$. This allows us to define a bounded linear operator $G_\Psi: (\widetilde H^s(\Omega))^* \rightarrow \widetilde H^s(\Omega)$ associating elements of $(\widetilde H^s(\Omega))^*$ to their Riesz representatives in said inner product. Eventually, we let the solution be $\widetilde u:=G_\Psi(\widetilde F)\in \widetilde H^s(\Omega)$. The remaining claims follow by the spectral theorem for the compact induced operator $\widetilde G_\Psi: L^2(\Omega) \rightarrow L^2(\Omega)$.
\end{proof}

The well-posedness for our problem is thus granted as soon as $0\not\in\Sigma$. We observe that the set $\Sigma$ of the eigenvalues depends on $\Psi$. Thus from now on we shall assume that $\Psi$ is such that 
\begin{equation}\label{eigenvalues}
    \mbox{if } u\in H^s(\mathbb R^n) \mbox{ solves } (-\Delta)^su+\Psi u =0 \mbox{ in } \Omega \mbox{ and } u|_{\Omega_e}=0\;, \mbox{ then } u\equiv 0\;.
\end{equation} 
With this in mind, we can define the Poisson operator $P_\Psi : H^s(\mathbb R^n)\rightarrow H^s(\mathbb R^n)$ associating each exterior value $f\in H^s$ to the unique solution $u$ to \eqref{directproblem} corresponding to it. Here we assume $F\equiv 0$. 
 
\begin{Rem}
\label{wellposedadjoint} 
By Lemma \ref{quasiprops}, everything we have proved above for the perturbation $\Psi$ holds for $\Psi^*$ as well. If we assume that the statement in \eqref{eigenvalues} also holds for $\Psi^*$, that is 
\begin{equation*}
    \mbox{if } u\in H^s(\mathbb R^n) \mbox{ solves } (-\Delta)^su+\Psi^* u =0 \mbox{ in } \Omega \mbox{ and } u|_{\Omega_e}=0\;, \mbox{ then } u\equiv 0\;,
\end{equation*}
then the Poisson operator $P_{\Psi^*}: H^s(\mathbb R^n)\rightarrow H^s(\mathbb R^n)$ is also well-defined.
\end{Rem}

Recall the definition of the abstract space of exterior values $X:= H^s(\mathbb R^n)/\widetilde H^s(\Omega)$. It can of course be equipped with the usual quotient norm. The elements of $X$ are equivalence classes of $H^s$ functions which coincide on $\Omega$, and thus it is natural to think of $X$ as $H^s(\Omega_e)$. When the boundary of $\Omega$ is Lipschitz, the identification $X=H^s(\Omega_e)$ can actually be rigorously shown.
\vspace{3mm}

Given the well-posedness of the direct problem, we can turn to the definition of the DN maps. This is done in the following proposition.

\begin{Pro}
\label{DNmaps}
Let $\Omega\subset\mathbb R^n$ be a bounded open set. Let $s\in(0,1)$, and assume $\Psi\in B_0(H^s,H^{-s})$. There exist two continuous linear maps
$$\Lambda_{\Psi}: X \rightarrow X^* \quad \mbox{defined by} \quad \langle \Lambda_\Psi[f],[g] \rangle:= B_\Psi(P_\Psi f,g)$$
and
$$\Lambda_{\Psi^*}: X \rightarrow X^* \quad \mbox{defined by} \quad \langle \Lambda_{\Psi^*}[f],[g] \rangle:= B_{\Psi^*}(P_{\Psi^*} f,g)\;.$$
Moreover, we have $(\Lambda_\Psi)^* = \Lambda_{\Psi^*}$.
\end{Pro}

\begin{proof}
The proof goes along the same lines as the analogous statements in \cite{GSU20} and \cite{CMRU20}. Observe that if $[f]\in X$, then $[f]=f+\widetilde H^s(\Omega)$, and thus by well-posedness $P_\Psi f' = P_\Psi f$ for all $f'\in[f]$. Moreover, $B_\Psi(P_\Psi f,g')=B_\Psi(P_\Psi f,g)$ for all $g'\in [g]$ by definition of $P_\Psi$. Therefore, the DN map $\Lambda_\Psi$ is well-defined. Its continuity is an immediate consequence of the boundeness Lemma \ref{boundedness}. Given Lemma \ref{quasiprops} and Remark \ref{wellposedadjoint}, the DN map $\Lambda_{\Psi^*}$ is also well-defined, linear and continuous. 
\vspace{3mm}

To see the last equality, observe that $P_{\Psi}g \in [g]$ and $P_{\Psi^*}f \in [f]$  for all $f,g\in H^s(\mathbb R^n)$. Therefore, using the definition of the bilinear form,
\begin{align*}
\langle\Lambda_{\Psi^*}[f],[g] \rangle & = \langle\Lambda_{\Psi^*}[f],[P_{\Psi}g] \rangle = B_{\Psi^*}(P_{\Psi^*} f,P_{\Psi}g) \\ & = \langle (-\Delta)^{s/2} P_{\Psi^*} f, (-\Delta)^{s/2}P_{\Psi}g \rangle + \langle \Psi^* P_{\Psi^*} f,P_{\Psi}g \rangle \\ & = \langle (-\Delta)^{s/2}P_{\Psi}g,(-\Delta)^{s/2} P_{\Psi^*} f \rangle + \langle  \Psi P_{\Psi}g,P_{\Psi^*} f \rangle \\ & = B_\Psi (P_\Psi g, P_{\Psi^*}f) = \langle \Lambda_\Psi [g], [P_{\Psi^*}f] \rangle = \langle \Lambda_\Psi [g], [f] \rangle = \langle  (\Lambda_\Psi)^*[f],[g] \rangle\;.
\end{align*}
\end{proof}

\section{The UCP and Runge approximation property for quasilocal perturbations}\label{sec:UCP-Runge}
\subsection{The unique continuation property}
In this section we are interested in the study of the unique continuation property (UCP) and Runge approximation property for the perturbed Schr\"odinger operator $(-\Delta)^s+\Psi$, where $\Psi$ in quasilocal. Here we do not make use of the assumption that $\Psi\in B(H^s,H^{-s})$ can be approximated by bounded operators of $L^2$ on itself. 

\begin{The}[UCP for $(-\Delta)^s + \Psi$, $\Psi$ of finite propagation]\label{th:finite-propagation} Let $\Omega,W\subset\mathbb R^n$ be non-empty open sets such that $u\in\widetilde H^s(\Omega)$, $((-\Delta)^su + \Psi u)|_W=0$ and $W \cap N(\Omega,p(\Psi))_e\neq \emptyset$. Then $u\equiv 0$.
\end{The}

\begin{proof}
Let $W':= W \cap N(\Omega,p(\Psi))_e$. 
Given that $u\in \widetilde H^s(\Omega)$, we have $u|_{W'}=0$. Moreover, since $\Psi$ has finite propagation, we have that spt$(\Psi u) \subseteq N(\spt u,p(\Psi)) \subseteq N(\Omega,p(\Psi))$. Therefore it must be $(\Psi u)|_{N(\Omega,p(\Psi))_e}=0$. Then we have $$ (-\Delta)^s u =0\;,\quad u=0\quad\quad \mbox{in } W'\;. $$ Also, $W'$ is non-empty by the geometric assumption and open by definition. Now the UCP for the fractional Laplacian (Lemma \ref{ucpfraclap}, see also \cite{GSU20}) gives $u\equiv 0$.
\end{proof}

In what follows we shall adapt the proof of Theorem \ref{th:finite-propagation} to the case when $\Psi$ is quasilocal. An obstacle of geometric nature arises immediately: the set $W\cap N(\Omega,r)_e$ should now be non-empty for all $r>0$, that is, $W$ should be unbounded. Moreover, since the quasilocality condition, as opposed to the finite propagation one, is given in the form of an inequality, we need to make use of the quantitative version of the UCP result for the fractional Laplacian from \cite{RS17}. This has the effect of introducing a condition on the decay function $\mu$. 
\vspace{3mm}

Before proceeding to the statement and proof of the UCP for quasilocal perturbations, we need to perform the following construction. Assume $W\subset\mathbb R^n$ is an unbounded open set, $\Omega\subset W_e$ is a bounded, open and smooth set, $u\in \widetilde H^s(\Omega)$ and $r>0$ is a fixed real number. Define $W_r:= W\cap N(\Omega,r)_e$ for $r>0$. By the geometric assumption, the set $W_r$ is known to be open and non-empty. Since spt$(u)\subseteq \Omega$, we can observe that $W_r \subseteq W \cap N(\spt u,r)_e$. In particular, $W_r \subseteq W\subset \Omega_e$. 
\noindent Fix now any $M\geq0$. In correspondence to $W_r,\Omega$ and $M$ we can find the two constants $c_M(r),\sigma_M(r)$ from Theorem \ref{logestimate} (check the Appendix for their exact values). Upon increasing $r$, the distance between $\Omega$ and $W_r$ also increases, which causes the estimate in Theorem \ref{logestimate} to become worse. This is quantified in the value of the constant $\sigma_M$, which verifies $\sigma_M(r)\rightarrow 0$ as $r\rightarrow\infty$. 


Next, we shall prove the UCP result for quasilocal perturbations:

 \begin{proof}[Proof of Theorem \ref{UCP_for_quasilocal_perturbation}]
Assume for the sake of contradiction that $u\not\equiv 0$. If we had $\|(-\Delta)^su\|_{L^2(\Omega)}=0$, then by the well-posedness of the Dirichlet problem for the fractional Schr\"odinger equation we could deduce $u\equiv 0$ (see condition (1.1) in \cite{GSU20} for $q=0$). Thus we can assume $\|(-\Delta)^su\|_{L^2(\Omega)}>0$.

\noindent Since $\Psi\in B(H^s,H^{-s})$ is quasilocal, it holds that $$\|\Psi u\|_{H^{-s}(W_{r})} \leq \|\Psi u\|_{H^{-s}(N(\spt u, r)_e)} \leq \mu(r)\|u\|_{H^{s}}\;.$$ Therefore by the assumption $((-\Delta)^su + \Psi u)|_W=0$ we get \begin{align*}
     \|(-\Delta)^s u\|_{H^{-s}(W_{r})} & \leq  \|(-\Delta)^s u + \Psi u\|_{H^{-s}(W_r)} +  \|\Psi u\|_{H^{-s}(W_r)} \\ & \leq \|(-\Delta)^s u + \Psi u\|_{L^{2}(W)}+\|\Psi u\|_{H^{-s}(W_r)}\leq \mu(r)\|u\|_{H^s}\;.
\end{align*} 
Let $u_1 := \frac{u}{\|(-\Delta)^su\|_{L^2(\Omega)}}$. Then $u_1\in \widetilde H^s(\Omega)$, and it solves $(-\Delta)^su_1 = v$ in $\Omega$, with $\|v\|_{L^2(\Omega)} = 1$. By the quantitative UCP for the fractional Laplacian (i.e. our Theorem \ref{logestimate} for $\Psi=0$) and the well-posedness of the direct problem for the fractional Schr\"odinger equation, we can estimate
\begin{equation*}\begin{split} 
  \|u_1\|_{H^s} & \leq C\,\|v\|_{H^{-s}(\Omega)} \leq C\,c_M(r) |\log (\|(-\Delta)^s u_1\|_{H^{-s}(W_{r})})|^{-\sigma_M(r)} \;,
\end{split}\end{equation*}
for all $M\geq 0$, while the previous computations now give
$$ \|(-\Delta)^s u_1\|_{H^{-s}(W_{r})} = \frac{\|(-\Delta)^s u\|_{H^{-s}(W_{r})}}{\|(-\Delta)^su\|_{L^2(\Omega)}} \leq \mu(r)\frac{\| u\|_{H^s}}{\|(-\Delta)^su\|_{L^2(\Omega)}} = \mu(r) \|u_1\|_{H^s} \;.$$
Therefore
$$ \|u_1\|_{H^s}\leq C\, c_M(r)|\log (\mu(r)\|u_1\|_{H^s})|^{-\sigma_M(r)} = C\, c_M(r)|\log \mu(r) + \log\|u_1\|_{H^s}|^{-\sigma_M(r)}\;. $$
Taking the limit $r\rightarrow \infty$, we observe that it must be $\|u_1\|_{H^s}=0$ by the strong decay assumption on $\mu$. We have thus obtained that $u_1\equiv 0$, which entails $u\equiv 0$.
\end{proof}

\begin{Rem}
Condition \eqref{mu-condition} for $\mu$ depends on the particular choice of sets $\Omega, W$ through the functions $c,\sigma$.
It is always trivially verified by perturbations $\Psi$ of finite propagation, since in this case $\mu(r)=0$ for all $r>p(\Psi)$. However, as proved in the examples of Section \ref{sec:examples}, for every couple of sets $\Omega, W$ as in Theorem \ref{th:quasilocal} it is always possible to find perturbations $\Psi$ which are quasilocal, not of finite propagation, and which verify \eqref{mu-condition}. 
\end{Rem}

\begin{Rem}
The result of Theorem \ref{th:finite-propagation} can be generalized to $s\in \mathbb R^+\setminus \mathbb Z$ by making use of the higher order fractional Poincar\'e inequality and UCP from \cite{CMR20}. We expect a similar generalization to hold in the case of Theorem \ref{th:quasilocal} as well; however, this would first require a quantitative UCP result for the higher order fractional Laplacian. 
\end{Rem}

\begin{Rem}\label{landis}
We believe the unique continuation result contained in Theorem \ref{th:quasilocal} is interesting by itself, as it may be related in its methods to the fractional Landis conjecture studied in \cite{RW19}. In such paper, the authors show that a function $u\in\widetilde H^s(\mathbb R^n)$ solving $(-\Delta)^su+qu=0$ in $\mathbb R^n$ must identically vanish as soon as the bounded potential $q$ verifies some regularity conditions and the decay of $u$ at infinity is strong enough, in the sense that there exists $\alpha>1$ such that $\int_{\mathbb R^n}e^{|x|^\alpha}|u|^2dx \leq C <\infty$. Given the boundedness of $q$, this can also be interpreted as a sufficient decay condition for the perturbation $qu$. In our Theorem \ref{th:quasilocal}, the function $u$ is in $\widetilde H^s(\Omega)$, but it is assumed to solve $(-\Delta)^su+\Psi u=0$ in $W\subset \Omega_e$. This condition is much less restrictive, given that it only applies to $W$ rather than $\mathbb R^n$, and the perturbation $\Psi u$ is more general than the term $qu$. The decay of $\Psi u$ in our case is prescribed by the quasilocality condition and assumption \eqref{mu-condition}. As shown in Example 3 from Section \ref{sec:examples}, 
$$  \|\Psi u\|_{H^{-s}(N(\Omega,r)_e)} \leq \mu(r)\|u\|_{H^s} = e^{-(c(r)f(r))^{1/\sigma(r)}}\|u\|_{H^s} \leq e^{- C f(r)^\beta}\|u\|_{H^s}$$
must hold for some function $f: \mathbb R^+ \rightarrow \mathbb R^+_0$ with $\lim_{r\rightarrow \infty}f(r)=\infty$. Here we were able to find $C, \beta>0$ independent of $r$ because of the known behaviours of $c(r), \sigma(r)$ as $r\rightarrow \infty$ (check the Appendix). Thus we see that in both our result and the one in \cite{RW19} the vanishing of a solution $u$ is obtained in the case of sufficient exponential decay of the involved perturbation.  
\end{Rem}

\subsection{The Runge approximation property}
Making use of the UCP results from Theorems \ref{th:finite-propagation} and \ref{th:quasilocal}, we now prove Runge approximation for perturbations which are either of finite propagation or quasilocal. 

\begin{Pro}[Runge approximation for $\Psi$ of finite propagation]\label{Runge-finite-propagation}
Let $W\subset \mathbb R^n$ be open and $\Omega\subset W_e$ be open and bounded with the property that $W \cap N(\Omega,p(\Psi))_e\neq \emptyset$. Then the set
$$ \mathcal{R}:= \{ P_\Psi f-f : f\in C^\infty_c(W) \} $$
is dense in $\widetilde H^s(\Omega)$.
\end{Pro}

\begin{proof}
By the Hahn-Banach theorem, it will suffice to prove that any $F\in (\widetilde H^s(\Omega))^*$ such that $\langle F,v\rangle=0$ for all $v\in \mathcal R$ must necessarily vanish. Given $F$ with the required properties, let $\phi\in \widetilde H^s(\Omega)$ be the unique solution to
\begin{equation*}
\begin{array}{rll}
      (-\Delta)^s\phi+\Psi^* \phi  & = -F & \mbox{ in } \Omega \\
      \phi  &= 0 &  \mbox{ in } \Omega_e
\end{array}\;.
\end{equation*}
Then for all $f\in C^\infty_c(W)$ we have
$$ 0 = \langle F,v\rangle = -B^*_\Psi(\phi, P_\Psi f-f) = B^*_\Psi(\phi,f) = \langle (-\Delta)^s\phi+\Psi^*\phi, f \rangle\;,  $$
which by the arbitrariety of $f$ implies $((-\Delta)^s\phi+\Psi^*\phi)|_W=0$. Since the definition of $\phi$ already gives us $\phi|_W=0$, by the geometric assumptions and Theorem \ref{th:finite-propagation} we are allowed to deduce $\phi\equiv 0$ in $\mathbb R^n$. Therefore we conclude that $F$ vanishes.
\end{proof}

\begin{Pro}[Runge approximation for $\Psi$ quasilocal]\label{Runge-quasilocal}
Let $W\subset \mathbb R^n$ be an unbounded open set, and let $\Omega\subset W_e$ be a bounded, open and smooth set. Assume that the decay function $\mu$ of the perturbation $\Psi\in B(L^2,L^2)$ verifies condition \eqref{mu-condition} for some $M\geq 0$. Then the set
$$ \mathcal{R}:= \{ P_\Psi f-f : f\in C^\infty_c(W) \} $$
is dense in $L^2(\Omega)$.
\end{Pro}

\begin{proof}
It follows from Lemma \ref{quasiprops} that $\Psi$ has decay function $\mu$ when regarded as an operator in $B(H^s,H^{-s})$, and the same holds true for its adjoint $\Psi^*\in B(H^s,H^{-s})$.
Given this observation, the rest of the proof goes along the same lines as the previous one. By Hahn-Banach, it is enough to prove that any $F\in L^2(\Omega)$ such that $\langle F,v \rangle$=0 for all $v\in \mathcal R$ vanishes. Considering again $\phi\in \widetilde H^s(\Omega)$ as above, we are lead to $((-\Delta)^s\phi + \Psi^*\phi)|_W = 0$. We also know that
$$ \|(-\Delta)^s\phi\|_{L^2(\Omega)}=\|F+\Psi^*\phi\|_{L^2(\Omega)}\leq \|F\|_{L^2(\Omega)} + \|\Psi^*\phi\|_{L^2} \leq \|F\|_{L^2(\Omega)} + \|\Psi^*\|\,\|\phi\|_{L^2} < \infty \;,$$ because $\Psi^*\in B(L^2,L^2)$ and $\phi\in\widetilde H^s(\Omega)\subset L^2$. Now the UCP result for quasilocal perturbations gives $\phi\equiv 0$, which entails $F\equiv 0$.
\end{proof}

\section{Quantitative Runge approximation}\label{sec:quantiRunge}

In this section we refine our result on Runge approximation for quasilocal perturbations by means of a quantitative estimate. This is reminiscent of Theorems 1.3 and 1.4 from \cite{RS17}.

\begin{The}[Quantitative Runge approximation property]\label{Qrunge}
Let $\Omega\subset\mathbb R^n$ be a bounded, smooth and open set, and let $W\subset\Omega_e$ be unbounded, smooth and open. Assume $s\in (0,1), M>0$ and that $\Psi\in B(L^2,L^2)$ is a quasilocal perturbation with $\|\Psi\|\leq M$ whose decay function $\mu$ satisfies condition \eqref{mu-condition}. Let $v\in \widetilde H^s(\Omega)$ and $\epsilon, \delta \in (0,1)$. There exist a real number $R_\epsilon >0$ and a function $f_\epsilon\in H^s_{\overline W_{R_\epsilon}}$ such that
$$ \|P_\Psi f_\epsilon - v\|_{L^2(\Omega)}\leq \epsilon \|v\|_{H^s_{\overline \Omega}} \qquad \mbox{and} \qquad \|f_{\epsilon}\|_{H^s_{\overline W_{R_\epsilon}}}\leq \mu(R_\epsilon)^{\delta-1}\|v\|_{L^2}\;. $$
\end{The}

\begin{proof} \textbf{Step 1.} Fix $\epsilon, \delta \in (0,1)$, and let $C>0$ be the constant from the well-posedness estimate. Since by assumption $\lim_{\rho\rightarrow \infty} c_M(\rho)|\log(\mu(\rho))|^{-\sigma_M(\rho)}=0$, it is possible to find $R_\epsilon>0$ so large that \begin{equation}\label{mu-condition-delta}
    c_M(R_\epsilon)|\log(\mu(R_\epsilon)^{1-\delta}+C\mu(R_\epsilon))|^{-\sigma_M(R_\epsilon)}<\epsilon\;.
\end{equation} From now on, we shall fix $W':= W_{R_\epsilon}$.
\vspace{3mm}

\noindent\textbf{Step 2.} Define the operator $A:= jr_\Omega P_\Psi : H^s_{\overline W'} \rightarrow L^2(\Omega)$, where $j: H^s(\Omega)\rightarrow L^2(\Omega)$ is the inclusion map, which is compact by Sobolev embedding. The well-posedness result (Proposition \ref{wellposed}) implies the boundedness of $r_\Omega P_\Psi$, so we have that $A$ is compact by composition.

Consider the Hilbert space adjoint $A^*$ of $A$, that is, the bounded (actually, compact) operator $A^* : L^2(\Omega)\rightarrow H^s_{\overline W'}$ defined by $\langle u,A^*v \rangle_{H^s_{\overline W'}}:=\langle Au,v\rangle_{L^2(\Omega)} $. Then $A^*A:H^s_{\overline W'}\rightarrow H^s_{\overline W'}$ is itself compact as composition of compact operators, and also positive and self-adjoint. The spectral theorem applies, and one can find $\{\lambda_j\}_{j\in\mathbb N}\subset \mathbb R^+$ with $\lambda_1 \geq \lambda_2 \geq ... \rightarrow 0 $ and an orthonormal basis $\{\phi_j\}_{j\in\mathbb N}$ of eigenfunctions of $H^s_{\overline W'}$ with $A^*A \phi_j = \lambda_j \phi_j$ for all $j\in\mathbb N$. If $w_j := \frac{A\phi_j}{\lambda_j^{1/2}}$, then $$\langle w_j, w_i \rangle_{L^2(\Omega)} = \frac{\langle A\phi_j, A\phi_i \rangle_{L^2(\Omega)}}{(\lambda_i\lambda_j)^{1/2}} = \frac{\langle \phi_j, A^*A\phi_i \rangle_{H^s_{\overline W'}}}{(\lambda_i\lambda_j)^{1/2}} = \frac{\lambda_i^{1/2}}{\lambda_j^{1/2}}\langle \phi_j, \phi_i \rangle_{H^s_{\overline W'}} = \delta_{ij}\;.$$ Moreover, if $u\in H^s_{\overline W'}$ and $h\in L^2(\Omega)$ is such that $\langle h,w_j\rangle_{L^2(\Omega)}=0$ for all $j\in\mathbb N$, then
$$\langle h, Au\rangle_{L^2(\Omega)} = \langle h, A\sum_{j\in\mathbb N} \langle u,\phi_j \rangle_{H^s_{\overline W'}} \phi_j \rangle_{L^2(\Omega)} = \sum_{j\in\mathbb N} \lambda_j^{1/2}\langle u,\phi_j \rangle_{H^s_{\overline W'}}\langle h, w_j \rangle_{L^2(\Omega)}=0\;.$$ Given that $A$ has dense range because of the Runge approximation property (Proposition \ref{Runge-quasilocal}), the last computation implies $h\equiv 0$ by Hahn-Banach. Therefore, $\{w_j\}_{j\in\mathbb N}$ is a complete orthonormal set in $L^2(\Omega)$.
\vspace{3mm}

\noindent \textbf{Step 3.} Let $v\in \widetilde H^s(\Omega)$. Define $N_\epsilon := \mu(R_\epsilon)^{1-\delta}$ and
$$f_\epsilon := \sum_{j:\lambda_j>N^2_\epsilon}\frac{\langle v, w_j \rangle_{L^2(\Omega)}}{ \lambda_j^{1/2}}\phi_j\;, \qquad  r_\epsilon := \sum_{j:\lambda_j\leq N^2_\epsilon} \langle v, w_j \rangle_{L^2(\Omega)} w_j\;.$$
Observe that $Af_\epsilon = \sum_{j:\lambda_j>N^2_\epsilon} \langle v, w_j \rangle_{L^2(\Omega)}w_j$. Also by orthonormality we get
$$ \|A^*r_\epsilon\|_{H^s_{\overline W'}} = \|\sum_{j:\lambda_j\leq N^2_\epsilon} \lambda_j^{1/2}\langle v, w_j \rangle_{L^2(\Omega)}\phi_j \|_{H^s_{\overline W'}} = (\sum_{j:\lambda_j\leq N^2_\epsilon} \lambda_j \langle v, w_j \rangle_{L^2(\Omega)}^2)^{1/2} \leq N_\epsilon \|r_\epsilon\|_{L^2(\Omega)}$$
and
\begin{align*}\| Af_\epsilon - v \|^2_{L^2(\Omega)} &= \| \sum_{j:\lambda_j\leq N^2_\epsilon} \langle v, w_j \rangle_{L^2(\Omega)}w_j \|^2_{L^2(\Omega)} = \sum_{j:\lambda_j\leq N^2_\epsilon} \langle v, w_j \rangle_{L^2(\Omega)}^2 \\ & = \langle v, \sum_{j:\lambda_j\leq N^2_\epsilon} \langle v, w_j \rangle_{L^2(\Omega)} w_j \rangle_{L^2(\Omega)} = \langle v, r_\epsilon \rangle_{L^2(\Omega)} \leq \|v\|_{H^s_{\overline \Omega}}\|r_\epsilon\|_{H^{-s}(\Omega)}\;.  \end{align*}
Finally, observe that $\|r_\epsilon\|_{L^2(\Omega)}^2= \sum_{j:\lambda_j\leq N^2_\epsilon} \langle v, w_j \rangle_{L^2(\Omega)}^2 = \| Af_\epsilon - v \|^2_{L^2(\Omega)} $.
\vspace{3mm}

\noindent \textbf{Step 4.} Let $u_\epsilon \in \widetilde H^s(\Omega)$ be the unique solution of the problem $$(-\Delta)^su_\epsilon+\Psi^* u_\epsilon = \frac{r_\epsilon}{\|r_{\epsilon}\|_{L^2(\Omega)}} \quad\mbox{ in } \Omega\;,\qquad \mbox{with } u_\epsilon|_{\Omega_e}=0\;.$$
By Theorem \ref{logestimate} we have $\|r_\epsilon\|_{H^{-s}(\Omega)}\leq c_M(R_\epsilon)\|r_\epsilon\|_{L^2(\Omega)}|\log (\|(-\Delta)^su_\epsilon\|_{H^{-s}(W')})|^{-\sigma_M(R_\epsilon)} $, since $\|\Psi^*\|=\|\Psi\|\leq M$. On the other hand, for all $g\in H^s_{\overline W'}$ it holds
\begin{align*}
    \langle A^*r_\epsilon, g \rangle_{H^s_{\overline W'}}  & = \langle r_\epsilon, Ag \rangle_{L^2(\Omega)} = \|r_\epsilon\|_{L^2(\Omega)}\langle (-\Delta)^su_\epsilon+\Psi^* u_\epsilon, Ag \rangle_{L^2(\Omega)} \\ & = \|r_\epsilon\|_{L^2(\Omega)}\langle (-\Delta)^su_\epsilon+\Psi^* u_\epsilon, P_\Psi g-g \rangle \\ & = \|r_\epsilon\|_{L^2(\Omega)}\langle u_\epsilon, ((-\Delta)^s+\Psi)P_\Psi g \rangle - \|r_\epsilon\|_{L^2(\Omega)}\langle (-\Delta)^su_\epsilon+\Psi^* u_\epsilon, g \rangle \\ & = - \|r_\epsilon\|_{L^2(\Omega)}\langle (-\Delta)^su_\epsilon+\Psi^* u_\epsilon, g \rangle\;,
\end{align*}
and thus $\|(-\Delta)^su_\epsilon + \Psi^* u_\epsilon\|_{H^{-s}(W')}\|r_\epsilon\|_{L^2(\Omega)} = \|A^*r_\epsilon\|_{H^s_{\overline W'}}$. We get
\begin{align*}
    \|r_\epsilon\|_{H^{-s}(\Omega)}\leq c_M(R_\epsilon)\|r_\epsilon\|_{L^2(\Omega)}\left|\log \left(\frac{\|A^*r_\epsilon\|_{H^s_{\overline W'}}}{\|r_\epsilon\|_{L^2(\Omega)}} + \|\Psi^* u_\epsilon\|_{H^{-s}(W')}\right)\right|^{-\sigma_M(R_\epsilon)}\;,
\end{align*} 
but since $\spt u_\epsilon \subset \Omega$ it holds that
$$ \|\Psi^* u_\epsilon\|_{H^{-s}(W')} \leq \|\Psi^* u_\epsilon\|_{H^{-s}(N(\spt(u_\epsilon),R_\epsilon)_e)} \leq \mu(R_\epsilon)\|u_\epsilon\|_{H^s}  \leq \frac{C\mu(R_\epsilon)\|r_\epsilon\|_{H^{-s}}}{\|r_\epsilon\|_{L^2(\Omega)}} \leq C\mu(R_\epsilon) $$
and so eventually 
\begin{align*}
    \|r_\epsilon\|_{H^{-s}(\Omega)}&\leq c_M(R_\epsilon)|\log (N_\epsilon + C\mu(R_\epsilon))|^{-\sigma_M(R_\epsilon)} \|r_\epsilon\|_{L^2(\Omega)} \\ & = c_M(R_\epsilon)|\log (\mu(R_\epsilon)^{1-\delta} + C\mu(R_\epsilon))|^{-\sigma_M(R_\epsilon)} \|r_\epsilon\|_{L^2(\Omega)} \leq \epsilon \|r_\epsilon\|_{L^2(\Omega)}\;.
\end{align*}

\noindent\textbf{Step 5.} Using the estimates in Steps 2 and 3, 
$$ \| Af_\epsilon - v \|_{L^2(\Omega)} = \frac{\| Af_\epsilon - v \|^2_{L^2(\Omega)}}{\|r_\epsilon\|_{L^2(\Omega)}} \leq \frac{\|v\|_{H^s_{\overline \Omega}}\|r_\epsilon\|_{H^{-s}(\Omega)}}{\|r_\epsilon\|_{L^2(\Omega)}} \leq \epsilon\|v\|_{H^s_{\overline \Omega}}\;. $$
This proves the first wanted inequality. For the second one, observe that
$$ \|f_\epsilon\|_{H^s_{\overline W'}}^2 = \|\sum_{j:\lambda_j>N^2_\epsilon}\frac{\langle v, w_j \rangle_{L^2(\Omega)}}{ \lambda_j^{1/2}}\phi_j\|_{H^s_{\overline W'}}^2 = \sum_{j:\lambda_j>N^2_\epsilon}\frac{\langle v, w_j \rangle^2_{L^2(\Omega)}}{ \lambda_j} \leq N_\epsilon^{-2} \|v\|_{L^2}^2\;, $$
and therefore $\|f_\epsilon\|_{H^s_{\overline W'}} \leq \mu(R_\epsilon)^{\delta-1}\|v\|_{L^2}$.
\end{proof}

\section{Proofs of the main theorems}\label{sec:proofs-main}
The proofs of our main theorems depend on the following Alessandrini identity:

\begin{Lem}\label{alex}
Let $\Omega\subset\mathbb R^n$ be a bounded open set. Let $s\in (0,1)$, and assume $\Psi_1, \Psi_2\in B_0(H^s,H^{-s})$. Then for all $f,g\in H^s(\mathbb R^n)$ it holds $$ \langle (\Lambda_{\Psi_1}- \Lambda_{\Psi_2})[f],[g] \rangle = \langle (\Psi_1-\Psi_2)P_{\Psi_1}f, P_{\Psi_2^*}g \rangle \;.$$
\end{Lem}

\begin{proof}
First of all, we observe that for $\Psi_1, \Psi_2\in B_0(H^s,H^{-s})$ both the Poisson operators and the DN maps are well-defined. Using the definitions and the property $(\Lambda_{\Psi_2})^* = \Lambda_{\Psi_2^*}$ from Proposition \ref{DNmaps}, we see that for all $f,g\in H^s(\mathbb R^n)$
\begin{align*}
    \langle (\Lambda_{\Psi_1}- \Lambda_{\Psi_2})[f],[g] \rangle & = \langle \Lambda_{\Psi_1}[f], [P_{\Psi_2^*}g] \rangle - \langle \Lambda_{\Psi_2}[P_{\Psi_1}f], [g] \rangle \\ & = \langle \Lambda_{\Psi_1}[f], [P_{\Psi_2^*}g] \rangle - \langle \Lambda_{\Psi_2^*}[g], [P_{\Psi_1}f]\rangle \\ & = B_{\Psi_1}(P_{\Psi_1}f, P_{\Psi_2^*}g) - B_{\Psi_2^*}(P_{\Psi_2^*}g, P_{\Psi_1}f) \\ & = \langle \Psi_1 P_{\Psi_1}f, P_{\Psi_2^*}g \rangle - \langle \Psi_2^* P_{\Psi_2^*}g , P_{\Psi_1}f  \rangle = \langle (\Psi_1-\Psi_2)P_{\Psi_1}f, P_{\Psi_2^*}g \rangle\;.
\end{align*}
\end{proof}

We are now ready to prove the main theorems.

\begin{proof}[Proof of Theorem \ref{main-finite}]
Let $p:= \max\{p(\Psi_1), p(\Psi_2)\}$, and assume the distances between $\Omega, W_1$ and $W_2$ are all at least $p$. In particular this implies $W_i\cap N(\Omega, p(\Psi_j))_e \neq \emptyset$ for all $i,j\in\{1,2\}$. Let $v_1,v_2\in \widetilde H^s(\Omega)$. Given the geometric conditions and Proposition \ref{Runge-finite-propagation}, for all $k\in\mathbb N$ it is possible to find $f_{j,k} \in C^\infty_c(W_j)$, $j=1,2$, such that
$$ P_{\Psi_1}f_{1,k} = f_{1,k}+v_1 +r_{1,k}\;, \qquad P_{\Psi_2^*}f_{2,k} = f_{2,k} + v_2  +r_{2,k} $$
with $r_{j,k} \rightarrow 0$ in $\widetilde H^s(\Omega)$ as $k\rightarrow \infty$. The assumption on the DN maps and the Alessandrini identity now give
\begin{align*}
    0 & = \langle (\Lambda_{\Psi_1}- \Lambda_{\Psi_2})[f_{1,k}],[f_{2,k}] \rangle = \langle (\Psi_1-\Psi_2)P_{\Psi_1}f_{1,k}, P_{\Psi_2^*}f_{2,k} \rangle \\ & = \langle (\Psi_1-\Psi_2)(f_{1,k}+v_1 +r_{1,k}), f_{2,k} + v_2  +r_{2,k} \rangle\;.
\end{align*}
Observe that $\spt(\Psi_i f_{j,k})\subset N( W_j,p )$ for all $i,j\in\{1,2\}$, and thus the equality above is reduced to
$$ \langle (\Psi_1-\Psi_2)(v_1 +r_{1,k}), v_2  +r_{2,k} \rangle =0\;. $$
Moreover,
$$ |\langle \Psi_i r_{j,k}, v \rangle| \leq \|\Psi_i\| \|r_{j,k}\|_{H^s} \|v\|_{H^s} \rightarrow 0 $$
for all $i,j\in\{1,2\}$ as $k\rightarrow \infty$. We have obtained $\langle (\Psi_1-\Psi_2)v_1 , v_2 \rangle =0$, which by the arbitrariety of $v_1, v_2 \in \widetilde H^s(\Omega)$ gives $\Psi_1 = \Psi_2$ as operators in $B(\widetilde H^s(\Omega),H^{-s}(\Omega))$.
\end{proof}

The proof of Theorem \ref{main-quasi} is similar to the one of the previous Theorem \ref{main-finite}, but we need the quantitative Runge approximation property from Theorem \ref{Qrunge} instead than the abstract result. This is due to the fact that now in general we do not have enough information on the supports.

\begin{proof}[Proof of Theorem \ref{main-quasi}] \textbf{Step 1.} Let $j=1,2$, and fix $\delta \in (1/2,1)$. Then for all $k\in\mathbb N$ it is possible to find a real number $R_{j,k}>0$ so large that condition \eqref{mu-condition-delta} holds for $\epsilon = k^{-1}$. If $R_k := \max_{j=1,2} R_{j,k}$, then of course both conditions hold at the same time.
\vspace{3mm}

\noindent \textbf{Step 2.} Let $v_1,v_2\in \widetilde H^s(\Omega)$ be such that $\|v_1\|_{H^s_{\overline \Omega}}=\|v_2\|_{H^s_{\overline \Omega}}=1$. Following the construction of Theorem \ref{Qrunge}, for all $k\in\mathbb N$ it is possible to find a function $f_{1,k} \in H^s_{\overline W_{1,R_{k}}}$ such that
$$ \|P_{\Psi_1} f_{1,k} - v_1\|_{L^2(\Omega)}\leq k^{-1}  \qquad \mbox{and} \qquad \|f_{1,k}\|_{H^s_{\overline W_{1,R_{k}}}}\leq \mu(R_{k})^{\delta-1}\;. $$
By density it is then possible to find $f_{1,k}'\in C^\infty_c(W_{1,R_{k}})$ such that $P_{\Psi_1} f'_{1,k} = f'_{1,k}+v_1+r_{1,k}$, with
$$ \|r_{1,k}\|_{L^2(\Omega)}= \|P_{\Psi_1} f'_{1,k} - v_1\|_{L^2(\Omega)}\leq c_1k^{-1}  \qquad \mbox{and} \qquad \|f'_{1,k}\|_{H^s_{\overline W_{1,R_{k}}}}\leq 2\mu(R_{k})^{\delta-1}\;, $$
where the constant $c_1>0$ depends only on $\Omega, \Psi_1$ and comes from the well-posedness estimate. We can do a similar construction for $v_2$ as well, which produces $f_{2,k}'\in C^\infty_c(W_{2,R_{k}})$ such that $P_{\Psi^*_2} f'_{2,k} = f'_{2,k}+v_2+r_{2,k}$, with 
$$ \|r_{2,k}\|_{L^2(\Omega)}=\|P_{\Psi_2^*} f'_{2,k} - v_2\|_{L^2(\Omega)}\leq c_2k^{-1}  \qquad \mbox{and} \qquad \|f'_{2,k}\|_{H^s_{\overline W_{2,R_{k}}}}\leq 2\mu(R_{k})^{\delta-1}\;. $$ Observe that by construction dist$(\Omega, \spt(f'_{j,k})) \geq R_{k}$, and therefore dist$(\spt(f'_{1,k}),\spt(f'_{2,k}))\geq R_{k}$ by the geometric assumption.
\vspace{3mm}

\noindent\textbf{Step 3.} The assumption on the DN maps and the Alessandrini identity again give
\begin{align*}
    \langle (\Psi_1-\Psi_2)(f'_{1,k}+v_1+r_{1,k}), f'_{2,k}+v_2+r_{2,k} \rangle =0\;.
\end{align*}

However, we can compute
\begin{align*} |\langle \Psi_j f'_{1,k}&, f'_{2,k} + v_2+r_{2,k}\rangle|  \leq \|\Psi_j f'_{1,k}\|_{L^2(\spt(f'_{2,k})\cup \Omega)}\| f'_{2,k}+v_2+r_{2,k} \|_{L^2} \\ & \leq \|\Psi_j f'_{1,k}\|_{L^2(N(\spt(f'_{1,k}),R_{k})_e)}\, (\| f'_{2,k}\|_{L^2(W_{2,R_{k}})}+ \| v_2 \|_{L^2(\Omega)}+ \| r_{2,k} \|_{L^2(\Omega)}) \\ & \leq \mu(R_{k}) \|f'_{1,k}\|_{L^2} \,( 2\mu(R_{k})^{\delta-1}+1+c_2k^{-1} ) \\ & \leq 2\mu(R_{k})^\delta  \,( 2\mu(R_{k})^{\delta-1}+1+c_2k^{-1} )\;,
\end{align*}
\noindent which vanishes as $k\rightarrow \infty$ since $\delta > 1/2$. Thus we are left with 
$$ \langle (\Psi_1-\Psi_2)(v_1+r_{1,k}), v_2+r_{2,k}\rangle=0\;, $$
and we estimate
\begin{align*}
    |\langle \Psi_jr_{1,k}, v_2+r_{2,k} \rangle| & \leq \| \Psi_jr_{1,k}\|_{L^2}\| v_2+r_{2,k}\|_{L^2} \\ & \leq \|\Psi_j\|\|r_{1,k}\|_{L^2(\Omega)}\, ( \| v_2\|_{L^2(\Omega)}+\|r_{2,k}\|_{L^2(\Omega)} ) \\ & \leq c_1 k^{-1}\|\Psi_j\|\, ( 1+c_2k^{-1} )\;,
\end{align*}
which also vanishes as $k\rightarrow \infty$. Eventually $ \langle (\Psi_1-\Psi_2)v_1, v_2\rangle=0 $ for all $v_1, v_2 \in \widetilde H^s(\Omega)$ of unitary norm. By rescaling, this implies that $\Psi_1=\Psi_2$ as operators in $B(\widetilde H^s(\Omega),H^{-s}(\Omega))$. 
\end{proof}

\begin{Rem}\label{geom-remark}
We remark at this point that we do not believe the geometric condition 
\begin{equation}\label{geom-cond-quasilocal}
    |x_1 - x_2| \geq \max_{j=1,2}\{ \mbox{dist}(x_j,\Omega) \}\qquad \mbox{ for all  } x_1\in W_1, x_2\in W_2
\end{equation} to be essential to the theorem. We have however decided to keep it, as it simplifies the proof to a certain extent and it is not the main focus of the discussion. It can be proved by simple geometric considerations that \eqref{geom-cond-quasilocal} is verified e.g. whenever $W_2\subseteq C^O(W_1)$, where the polar cone of $W_1$ is defined as
$$ C^O(W_1):= \{x\in \mathbb R^n : x\cdot w \leq 0, \forall w\in W_1\}\;, $$
and the origin $O$ is assumed to belong to $\Omega$. Observe that this condition is symmetric, that is, $W_2\subseteq C^O(W_1)$ if and only if $W_1\subseteq C^O(W_2)$. This in particular means that $W_1, W_2$ belonging to opposite orthants is an admissible choice (check Figure \ref{fig:figure1}). Since the sets upon which the DN map is measured can always be reduced, it is also clear that the stated geometric conditions need only be satisfied by aptly chosen \emph{subsets} of $W_1, W_2$. 
\end{Rem}

\begin{figure}
\centering
  \captionsetup{width=.8\linewidth}
  \includegraphics[width=\linewidth/2]{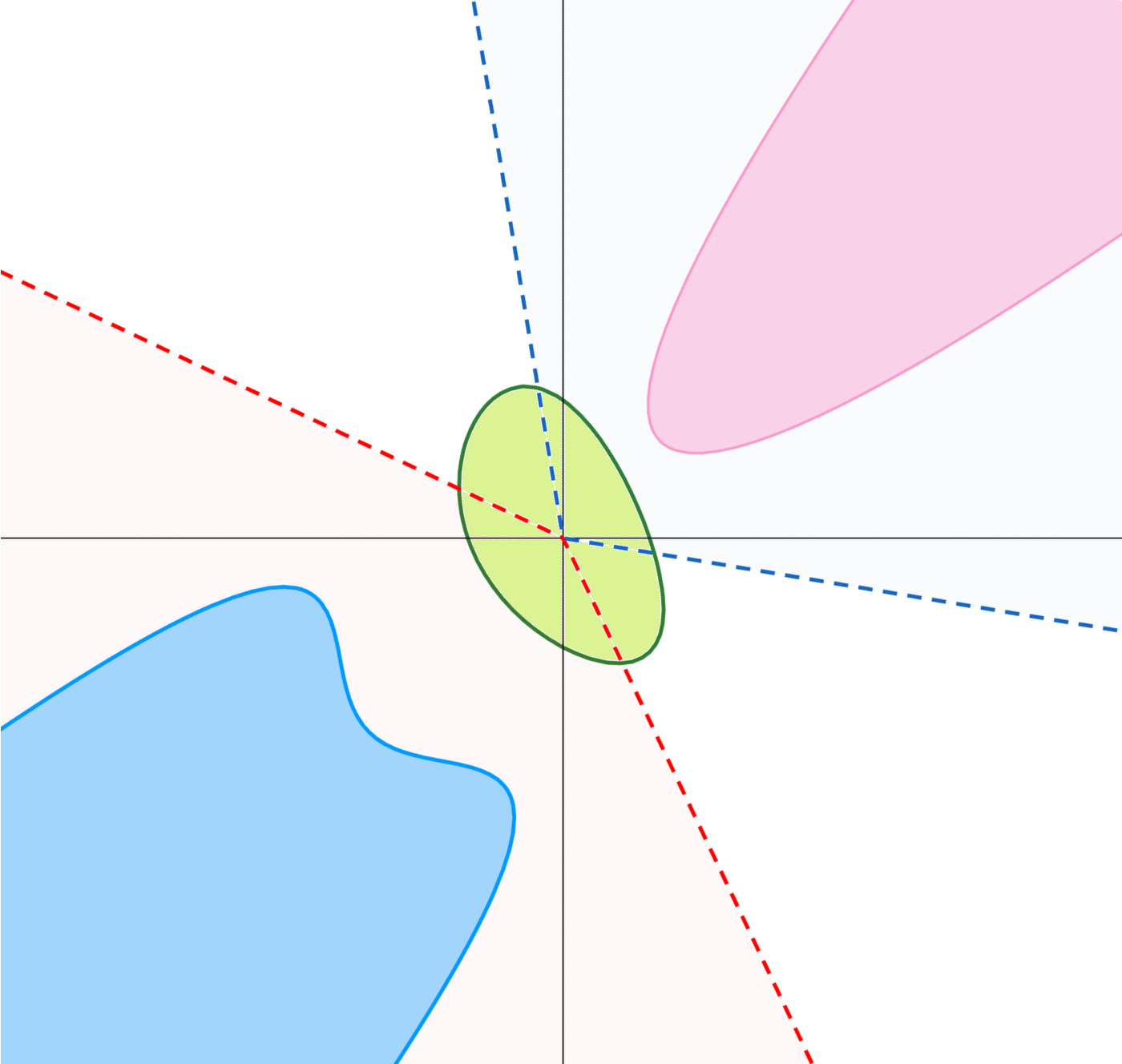}
  \caption{The geometric situation described in Remark \ref{geom-remark}. In green, pink and blue we respectively represent the sets $\Omega$, $W_1$ and $W_2$. The polar cones $C^O(W_1)$ and $C^O(W_2)$ are in pink and blue shades. Observe that $W_2\subseteq C^O(W_1)$ and $W_1\subseteq C^O(W_2)$. Therefore, condition \eqref{geom-cond-quasilocal} holds.}
  \label{fig:figure1}
\end{figure}

A linear operator between normed spaces is bounded if and only if it is continuous. Thus each element $\Psi \in B(\widetilde H^s(\Omega),H^{-s}(\Omega))$ can be interpreted as a continuous linear operator between $C^\infty_c(\Omega)$ and $\mathcal D'(\Omega)$, the set of all distributions supported on $\Omega$. By the Schwartz kernel theorem \cite{HO:analysis-of-pdos}, there exists one and only one distribution $\psi\in \mathcal D'(\Omega\times\Omega)$ such that 
$$  \langle \Psi u,v \rangle = \psi(u\otimes v)$$
holds for all $u,v\in C^\infty_c(\Omega)$. Such $\psi$ is called the kernel of $\Psi$, and if it is more regular one has the familiar integral formula
$$  \langle \Psi u,v \rangle = \int_{\mathbb R^n}\int_{\mathbb R^n} \psi(x,y)u(x)v(y)\, dxdy\;.$$
Thus we obtain the following Corollary to Theorems \ref{main-finite}, \ref{th:quasilocal}:
\begin{Cor}
Assume the conditions of either Theorem \ref{main-finite} or Theorem \ref{th:quasilocal} are satisfied. Then the kernels $\psi_1, \psi_2$ of the perturbations $\Psi_1, \Psi_2$ coincide as elements of $\mathcal D'(\Omega\times\Omega)$. If $\psi_1, \psi_2$ happen to be functions, then they coincide on $\Omega\times\Omega$.
\end{Cor}

\section{Examples}\label{sec:examples}

In this section we discuss a number of examples, in order to ensure that the assumptions of our theorems are verified by a large family of non-trivial perturbations.
\vspace{5mm}

\noindent \textbf{Example 1: Finite propagation operators $\Psi$ of given $p(\Psi)$.} Given any $R>0$, one can always construct operators $\Psi: L^2(\mathbb R^n)\rightarrow L^2(\mathbb R^n)$ such that $p(\Psi)=R$. In order to see this, fix e.g. any $\psi\in L^2(\mathbb R^n \times \mathbb R^n)$ such that $\spt\psi=\{(x,y)\in\mathbb R^{2n}:|x-y|\leq R\}$ with $\psi\geq 0$, and define the Hilbert-Schmidt operator $\Psi$ of kernel $\psi$ by
$$ \Psi u (x) := \int_{\mathbb R^n} \psi(x,y)u(y)dy\qquad \mbox{ for } u\in L^2(\mathbb R^n)\;. $$
Then famously $\|\Psi\|_{B(L^2,L^2)}\leq \|\psi\|_{L^2(\mathbb R^n \times \mathbb R^n)}<\infty$, so $\Psi\in B(L^2(\mathbb R^n), L^2(\mathbb R^n))$ (check e.g. \cite{HS78}). Moreover, if $x\in N(\spt u, R)_e$ and $y\in \spt u$ we have $|x-y|>R$, and so $\psi(x,y)=0$. Therefore $\Psi u(x) = \int_{\spt u} \psi(x,y)u(y)dy =0$, which implies $\spt (\Psi u)\subseteq N(\spt u, R)$. This proves that $\Psi$ has finite propagation as an operator in $B(L^2(\mathbb R^n),L^2(\mathbb R^n))$, and $p(\Psi)\leq R$. However, since $\Psi \chi_{B_1}(x) > 0$ for all $x\in B_{R+1}$, we must have $p(\Psi)=R$. 
\vspace{2mm}

This example shows in particular that the constructed operator $\Psi$ is nonlocal. It now follows from Lemma \ref{quasiprops} and approximation by smooth functions that, when interpreted as an operator in $B(H^s(\mathbb R^n),H^{-s}(\mathbb R^n))$, $\Psi$ has again finite propagation $R$.
\vspace{5mm}

\noindent \textbf{Example 2: Quasilocal operators $\Psi$ of given decay function $\mu$.} Given any function $\mu: \mathbb R^+\rightarrow \mathbb R^+$ such that $\lim_{r\rightarrow \infty}\mu(r)=0$, one can always construct a quasilocal operator $\Psi: L^2\rightarrow L^2$ admitting $\mu$ as a decay function. Of course, in order to avoid trivialities, we will construct $\Psi$ not of finite propagation. 
\vspace{2mm}

In order to see this, first of all observe that for a $\Psi\in B(L^2,L^2)$ generated by a kernel $\psi\in L^2(\mathbb R^n \times \mathbb R^n)$ we have 
\begin{align*}
    \|\Psi u\|^2_{L^2(N(\spt u,r)_e)} & = \int_{N(\spt u,r)_e}\left| \int_{\mathbb R^n} \psi(x,y)u(y)dy \right|^2dx \\ & \leq \int_{N(\spt u,r)_e} \left(\int_{\spt u} |\psi(x,y)u(y)|dy \right)^2dx \\ & \leq \|u\|^2_{L^2(\mathbb R^n)} \|\psi\|^2_{L^2(N(\spt u,r)_e\times \spt u)}\;.
\end{align*}
Since for all open sets $U\subset\mathbb R^n$ we have $N(U,r)_e\times U \subset \{ (x,y)\in\mathbb R^{2n}: y\in \mathbb R^n \setminus B_r(x) \}$, it suffices to take $\psi\in L^2(\mathbb R^n \times \mathbb R^n)$ such that $\psi>0$ and $ \int_{\mathbb R^n}\int_{\mathbb R^n\setminus B_r(x)} \psi^2(x,y) dydx \leq \mu^2(r)$ for all $r>0$. Assume that $\psi(x,y)=\psi_1(x)\psi_2(y-x)$ for some functions $\psi_1, \psi_2 : \mathbb R^n \rightarrow \mathbb R$, with $\|\psi_1\|_{L^2(\mathbb R^n)}=1$ and $\psi_2(z) = \widetilde \psi_2(|z|)$ radial. Then by a change of variable
\begin{align*}
    \int_{\mathbb R^n}\int_{\mathbb R^n\setminus B_r(x)}\psi^2(x,y)dydx & = \int_{\mathbb R^n}\psi^2_1(x)\int_{\mathbb R^n\setminus B_r(x)}\psi^2_2(y-x)dydx \\ & =  \|\psi_1\|^2_{L^2(\mathbb R^n)}\int_{\mathbb R^n\setminus B_r}\psi^2_2(z)dz = c_n \int_r^\infty\rho^{n-1} \widetilde \psi_2^2(\rho)d\rho\;.
\end{align*}
Let $\widetilde \mu\in C^\infty(\mathbb R^+)$ be strictly decreasing and such that $0<\widetilde \mu(r)\leq \mu(r)$ for all $r\in \mathbb R^+$. It suffices to find $\widetilde\psi_2$ such that
$$ c_n \int_r^\infty \rho^{n-1}\widetilde\psi_2^2(\rho)d\rho = \widetilde \mu^2(r) = -\int_r^\infty (\widetilde \mu^2)' (\rho)d\rho\;, $$
and we can take $\widetilde\psi_2(\rho) = \left(-\frac{ (\widetilde \mu^2)'(\rho)}{c_n \rho^{n-1}}\right)^{1/2}$. One also sees that $$\|\Psi\|\leq \|\psi\|_{L^2(\mathbb R^n \times \mathbb R^n)}= \|\psi_1\|_{L^2(\mathbb R^n)}\|\psi_2\|_{L^2(\mathbb R^n)}=\widetilde \mu(0)\;,$$
and thus the operator $\Psi$ can be taken of arbitrarily small norm.
\vspace{3mm}

If $\psi_1, \psi_2$ are strictly positive on $\mathbb R^n$, then $\spt (\Psi\chi_{B_1})=\mathbb R^n$, and thus the quasilocal operator $\Psi$ constructed in this example is not of finite propagation. Finally, using again Lemma \ref{quasiprops} one shows that $\Psi$ is also quasilocal with the same decay function when interpreted as an operator in $B(H^s,H^{-s})$.
\vspace{5mm}

\noindent \textbf{Example 3: Perturbations verifying the assumptions of Theorem \ref{main-quasi} } Given any $M>0$ and any couple of sets $\Omega, W$ as in Theorem \ref{main-quasi}, it is always possible to find quasilocal perturbations verifying $\|\Psi\|\leq M$ and condition \eqref{mu-condition} with respect to $M$. To see this, first observe that the fixed choice of $\Omega, W, M$ corresponds to determined functions $c,\sigma$. Given any $f:\mathbb R^+\rightarrow \mathbb R^+_0$ such that $\lim_{r\rightarrow \infty}f(r)=\infty$, the function $$\mu(r):= e^{-(c(r)f(r))^{1/\sigma(r)}}$$ verifies both \eqref{mu-condition} and $\lim_{r\rightarrow \infty}\mu(r)=0$. Using Example 2, one can then construct infinitely many different perturbations $\Psi$ admitting such $\mu$ as a decay function. As observed in the previous example, these can be arranged to have operator norm smaller than $M$.
\vspace{5mm}

\noindent \textbf{Example 4: Quasilocal perturbations as pseudodifferential operators } In this last example, we present some special cases in which our quasilocal perturbations turn out to be pseudodifferential operators. In this respect, we follow \cite{HO:analysis-of-pdos3}.

Given any $m\in\mathbb R$, the set $S^m$ of all $a\in C^\infty(\mathbb R^n \times \mathbb R^n)$ such that for all multi-indices $\alpha,\beta$ there is a positive constant $C_{\alpha,\beta}$ such that the estimate
$$ |\partial^\alpha_\xi\partial^\beta_x a(x,\xi)| \leq C_{\alpha,\beta}(1+|\xi|)^{m-|\alpha|} $$
holds for all $x,\xi \in \mathbb R^n$ is called the space of symbols of order $m$. If $a\in S^m$ is a symbol of order $m$ and $u\in \mathscr S$, then the operator defined by $$a(x,D)u(x):= \int_{\mathbb R^n}e^{ix\cdot\xi}a(x,\xi)\hat u(\xi)d\xi,$$
where $\hat u\in \mathscr S$ is the Fourier transform of $u$, is a pseudodifferential operator ($\Psi$DO) of order $m$. Such operator can alternatively be written as $$a(x,D)u(x):= \int_{\mathbb R^n} K(x,y)u(y)dy,$$
where the Schwartz kernel $K$ of the operator and its symbol $a$ are related in the following way:
$$ K(x,y)= \int_{\mathbb R^n}e^{i(x-y)\cdot\xi}a(x,\xi)d\xi, \qquad a(x,\xi)=\int_{\mathbb R^n}e^{-iy\cdot\xi}K(x,x-y)dy. $$

Assume now that $K(x,y):= k(x-y)$ for some function $k\in C^\infty_c(B_R)$ and $R>0$. In this case $\Psi$ is given by convolution $$\Psi u(x) = \int_{\mathbb R^n} k(x-y)u(y)dy = (k\ast u)(x),$$
and its symbol $a$ is
\begin{align*}
    a(x,\xi) = \int_{\mathbb R^n} e^{-iy\cdot \xi}k(y)dy = \mathcal F^{-1}k(\xi).
\end{align*}
Observe that $a$ does not depend on $x$, and moreover it is a Schwartz function by the regularity of $k$. Thus in particular $a\in S^0$, and $\Psi$ is a $\Psi$DO of order $0$. Using our previous Example 1, it is also easy to observe that $\Psi$ has finite propagation $R$. With this, we have proved that among all finite propagation perturbations there are in particular the convolutions against smooth, compactly supported functions, and these are pseudodifferential operators of order $0$.
\vspace{2mm}

It is also interesting to assume that $K(x,y):= k_1(x)k_2(x-y)$, where $k_1,k_2$ are Schwartz functions. Now $\Psi$ can be written as a convolution with a variable coefficient
$$\Psi u(x) = \int_{\mathbb R^n}k_1(x)k_2(x-y)u(y)dy = k_1(x)\left( k_2\ast u \right)(x),$$
and the symbol $a$ is a separable function of $x$ and $\xi$:
$$ a(x,\xi) = \int_{\mathbb R^n} e^{-iy\cdot\xi}k_1(x)k_2(y)dy = k_1(x)\mathcal F^{-1}k_2(\xi). $$
Since $k_1, \mathcal F^{-1}k_2$ are Schwartz functions and $\partial^\alpha_\xi\partial^\beta_x a(x,\xi) = (\partial^\beta_x k_1)(\partial^\alpha_\xi \mathcal F^{-1}k_2)$, the symbol $a(x,\xi)$ belongs to the class $S^0$. For any fixed $r>0$ we can compute
\begin{align*}
    \int_{\mathbb R^n}\int_{\mathbb R^n\setminus B_r(x)} K^2(x,y)dydx = \int_{\mathbb R^n}k_1^2(x)\int_{\mathbb R^n\setminus B_r(x)} k_2^2(x-y)dydx = \|k_1\|_{L^2(\mathbb R^n)}\|k_2\|_{L^2(\mathbb R^n\setminus B_r)}.
\end{align*}
Given that $k_1, k_2$ are Schwartz, the above quantity is finite for all $r>0$, and it vanishes as $r\rightarrow \infty$. According to our previous Example 2, $\Psi$ is then quasilocal, and it admits decay function $\mu(r) = (\|k_1\|_{L^2(\mathbb R^n)}\|k_2\|_{L^2(\mathbb R^n\setminus B_r)})^{1/2}$. With this, we have proved that among all quasilocal perturbations there are in particular the products of a Schwartz function and a convolution against a Schwartz function, and these are pseudodifferential operators of order $0$.

\section{Appendix}\label{sec:appendix}
In this Appendix we briefly (but very carefully) sketch the main ideas of the proof of Theorem 1.3 from \cite{RS17}. In doing so, we show how the constants $c,\sigma$ in our Theorem \ref{logestimate} depend on the set $W$ and parameter $M$.
\vspace{3mm}

Let $v\in L^2(\Omega)$ be such that $\|v\|_{L^2(\Omega)}=1$, and assume that $w,\widetilde w$ respectively solve the following problems:
\begin{equation*}
\begin{split}
    (-\Delta)^s w+\Psi w & =v \mbox{ in } \Omega \\
    w & =0 \mbox{ in } \Omega_e
\end{split}\;, \qquad \qquad
\begin{split}
    \nabla\cdot(x_{n+1}^{1-2s}\nabla \widetilde w) & =0 \mbox{ in } \mathbb R^{n+1}_+ \\
    \widetilde w(\cdot,0) & = w(\cdot) \mbox{ in } \mathbb R^n
\end{split}\;.
\end{equation*}

The function $\widetilde w$ is called the Caffarelli-Silvestre extension of $w$ (see \cite{CS14, CS07}), and it has the remarkable property that
$$ \lim_{x_{n+1}\rightarrow 0} x_{n+1}^{1-2s}\widetilde w(\cdot,x_{n+1}) = -a_s(-\Delta)^sw \qquad \mbox{ in } H^{-s}$$
for some constant $a_s>0$. Thanks to Lemma 4.2 in \cite{RS17} and our Proposition \ref{wellposed}, it holds that
\begin{equation}\label{cond1}
    \|x_{n+1}^{1/2-s}\widetilde w\|_{L^2(\mathbb R^n \times [0,2])} + \|x_{n+1}^{1/2-s}\nabla\widetilde w\|_{L^2(\mathbb R^{n+1}_+)} \leq c_{\Omega} \|w\|_{H^{s}(\mathbb R^n)} \leq c_{\Omega,M} \|v\|_{L^2(\Omega)} = c_{\Omega,M}=:E
\end{equation}
and
\begin{equation*}
    \|x_{n+1}^{1/2-s}\widetilde w\|_{L^2(W\times[0,1])} \leq \|(-\Delta)^s w\|_{H^{-s}(W)} =: \eta\;.
\end{equation*}
We observe that by taking $E$ large enough, one can have $\eta/E$ as small as desired (this will be useful later on). Moreover, by Vishik-Eskin estimates (\cite{Gr15}) and Lemmas 4.4, 6.2 in \cite{RS17} it holds that for $0\leq \gamma \leq \min\{s,1/2\}$ 
\begin{equation}\label{cond2}
    \|x_{n+1}^{1/2-s-\gamma}\nabla\widetilde w\|_{L^2(\mathbb R^{n+1}_+)} \leq c_{\Omega} \|w\|_{H^{s+\gamma}(\mathbb R^n)} \leq c'_{\Omega,M} \|v\|_{H^{\gamma-s}(\Omega)} \leq c'_{\Omega,M} \|v\|_{L^2(\Omega)} = c'_{\Omega,M}\;.
\end{equation}

Under conditions \eqref{cond1}, \eqref{cond2} it is known from Propositions 5.3, 5.4 and Theorem 5.5 in \cite{RS17} that there exist $K>1$, $\alpha \in (0,1)$ depending only on $n,s$ such that 
\begin{equation}\label{bound-bulk} \|x_{n+1}^{1/2-s}\widetilde w\|_{L^2(B_{2r}^+(x_0))} \leq KE^{1-\alpha}\|x_{n+1}^{1/2-s}\widetilde w\|_{L^2(B_{r}^+(x_0))}^\alpha \;, \end{equation}
where either $B'_{4r}(x_0)\subseteq W\times\{0\}$ with $x_0 \in W\times\{0\}$, or the $x_{n+1}$-component of $x_0$ is $5r$. Here we use the definitions $B':= B\cap (\mathbb R^n\times \{0\})$ and $B^+:= B\cap \mathbb R^{n+1}_+$, where $B$ is a $(n+1)$-dimensional ball. 
\vspace{3mm}

Since $\Omega$ is bounded, we certainly have $\Omega\subset B_{r_\Omega}(x_\Omega)$ for some point $x_\Omega\in\mathbb R^n$ and some $r_\Omega>0$. We also assume for the time being that $W$ is a ball of radius $r_W$ and center $x_W$, and assume $r_W \leq 2$. Then $B'_{r_W}(x_W) \subseteq W\times \{0\}$, and also $B_{r_W}^+(x_W)\subset \mathbb R^n\times[0,2]$. This implies that we can use estimate \eqref{bound-bulk} and write
\begin{align*}
    \|x_{n+1}^{1/2-s}\widetilde w\|_{L^2(B_{r_W/2}^+(x_W))} &\leq KE^{1-\alpha}\|x_{n+1}^{1/2-s}\widetilde w\|_{L^2(B_{r_W/4}^+(x_W))}^\alpha \\ & \leq KE^{1-\alpha}\|x_{n+1}^{1/2-s}\widetilde w\|_{L^2(W\times [0,1])}^\alpha \leq KE (\eta/E)^\alpha\;. 
\end{align*}

In order to be allowed to use \eqref{bound-bulk} for a center not belonging to $\mathbb R^n$, we need balls of radii depending on the $x_{n+1}$ component (or "height") of the center. We thus define $r(y):=y/5$ for all $y\in(0,2)$ to be the radius allowed for a ball with center at height $y$.

We now compute the largest height $y_1$ such that $B^+_{r(y_1)}(x_W,y_1) \subset B^+_{r_W/2}(x_W,0)$. This gives the relation $r(y_1)+y_1 = r_W/2$, which implies $y_1 := \frac{5}{12}r_W < \frac{5}{6}$ (check Figure \ref{fig:figure2}). Thus in particular
$$ \|x_{n+1}^{1/2-s}\widetilde w\|_{L^2(B_{r(y_1)}^+(x_W, y_1))} \leq \|x_{n+1}^{1/2-s}\widetilde w\|_{L^2(B_{r_W/2}^+(x_W))} \leq KE (\eta/E)^\alpha\;. $$

\begin{figure}
\centering
  \captionsetup{width=.45\linewidth}
  \includegraphics[width=.45\linewidth]{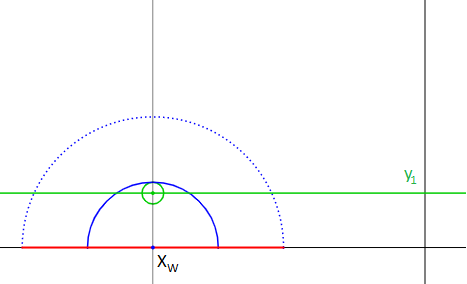}
  \caption{In solid and dashed blue, $B^+_{r_W/2}(x_W,0)$ and $B^+_{r_W}(x_W,0)$. In green, $B^+_{r(y_1)}(x_W,y_1)$. In red, $W$.} 
  \label{fig:figure2}
\end{figure}

Next we shall recursively find the largest height $y_{i+1}$ such that $B^+_{r(y_{i+1})}(x_W,y_{i+1})\subset B^+_{2r(y_{i})}(x_W,y_i)$. 
\begin{figure}[H]
\centering
  \captionsetup{width=.45\linewidth}
  \includegraphics[width=.45\linewidth]{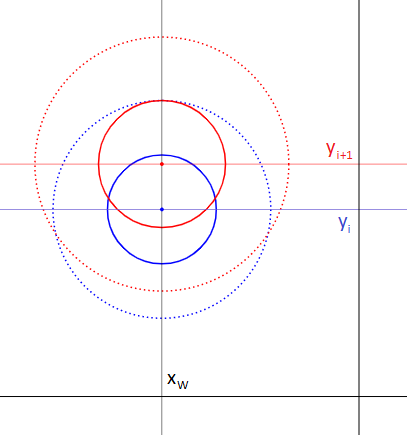}
  \caption{In solid blue and red, $B^+_{r(y_i)}(x_W,y_i)$ and $B^+_{r(y_{i+1})}(x_W,y_{i+1})$. In dashed blue and red, $B^+_{2r(y_i)}(x_W,y_i)$ and $B^+_{2r(y_{i+1})}(x_W,y_{i+1})$.} 
  \label{fig:figure3}
\end{figure}
From this one gets the relation $r(y_{i+1})+y_{i+1}= 2r(y_i)+y_i$, that is $y_{i+1}:= \frac{7}{6}y_i$ (check Figure \ref{fig:figure3}), and by \eqref{bound-bulk}
$$ \|x_{n+1}^{1/2-s}\widetilde w\|_{L^2(B_{r(y_{i+1})}^+(x_W, y_{i+1}))} \leq \|x_{n+1}^{1/2-s}\widetilde w\|_{L^2(B_{2r(y_{i})}^+(x_W, y_{i}))} \leq KE^{1-\alpha} \|x_{n+1}^{1/2-s}\widetilde w\|^\alpha_{L^2(B_{r(y_{i})}^+(x_W, y_{i}))} \;.  $$

Eventually $y_N = (\frac{7}{6})^{N-1} y_1$ for all $N\in\mathbb N$, and
\begin{align*} \|x_{n+1}^{1/2-s}\widetilde w\|_{L^2(B_{r(y_{N})}^+(x_W, y_{N}))} & \leq  (KE^{1-\alpha})^{\sum_{j=0}^{N-2}\alpha^j} \|x_{n+1}^{1/2-s}\widetilde w\|^{\alpha^{N-1}}_{L^2(B_{r(y_{1})}^+(x_W, y_{1}))} \\ & \leq K^{\frac{1-\alpha^{N-1}}{1-\alpha}}E^{1-\alpha^{N-1}} ( KE (\eta/E)^\alpha )^{\alpha^{N-1}} = K^{\frac{1 - \alpha^N}{1-\alpha}}E   (\eta/E)^{\alpha^N} 
\;.  \end{align*}
We require that $y_N = 1$, which gives $N= 1+\frac{\log y_1}{\log (6/7)} \leq 13|\log y_1|\leq 14|\log r_W|$ for $r_W$ small enough.
\vspace{3mm}

Consider the segment of endpoints $(x_W, 1)$ and $(x_\Omega, 1)$, and divide it into segments of length $\frac{1}{5}$ by a finite sequence of points $x_i$. Then of course $B^+_{1/5}(x_{i+1})\subset B^+_{2/5}(x_i)$ for every $i$, and we can apply \eqref{bound-bulk} for $N_1:=\frac{|x_W-x_\Omega|}{r(1)} = 5|x_W-x_\Omega| $ times. We get
\begin{align*} \|x_{n+1}^{1/2-s}\widetilde w\|_{L^2(B_{1/5}^+(x_\Omega,1))} & \leq  (KE^{1-\alpha})^{\sum_{j=0}^{N_1-2}\alpha^j} \|x_{n+1}^{1/2-s}\widetilde w\|^{\alpha^{N_1-1}}_{L^2(B_{1/5}^+(x_W, 1))} \leq  K^{\frac{1 - \alpha^{N+N_1}}{1-\alpha}}E   (\eta/E)^{\alpha^{N+N_1}} 
\;.  \end{align*}
We want to cover $\Omega\times\{1\}$ with balls of the kind $B^+_{1/5}(x,1)$, with $x\in \Omega\times\{1\}$. This can be done by choosing a finite number (depending only on $\Omega$) of "horizontal" directions and producing balls from $B_{1/5}^+(x_\Omega,1)$ as before. This requires at most $N_2:=\frac{r_\Omega}{r(1)}=5r_\Omega$ new steps and $N_\Omega$ new balls. We let their centers be $(x_k,1)$ with $k\in\{1,...,N_\Omega\}$.

Let $z_0:=1$. For each of the balls $B^+_{r(z_0)}(x_k,z_0)$ covering $\Omega\times\{1\}$ we shall identify $2^n$ balls at height $z_1$ and centers $(x_{k,1}, z_1), ..., (x_{k,2^n},z_1)$ such that $$B^+_{r(z_1)}(x_{k,j}, z_1)\subset B^+_{2r(z_0)}(x_k,z_0)\quad \mbox{ for all}\quad  j\in\{1,...,2^n\}$$ and $$ B'_{r(z_0)}(x_k,0)\times [z_1,z_0] \subseteq \bigcup_j B^+_{r(z_1)}(x_{k,j},z_1)\;.$$ 
Here the points $x_{k,j}$ can be taken to be the vertices of the $n$-cube inscribed in $B'_{r(z_0)/2}(x_k,0)$. In order to do so, it suffices to have $r(z_1)^2 \geq (z_0-z_1)^2 + \frac{r^2(z_0)}{4}$ (check Figure \ref{fig:figure4}), and after some computations one sees that this relation is verified e.g. if $z_1 = \frac{9}{10}z_0$. 
\begin{figure}
\centering
  \captionsetup{width=.6\linewidth}
  \includegraphics[width=.45\linewidth]{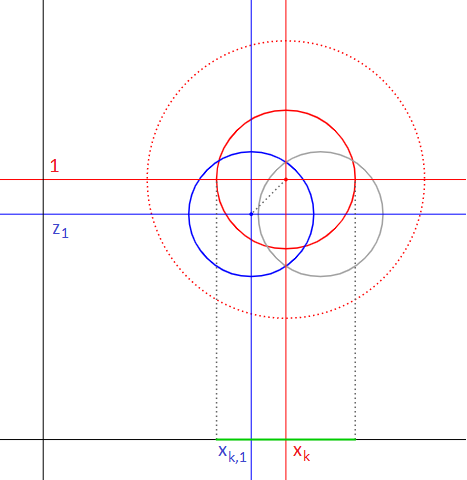}
  \caption{In solid and dashed red, the balls $B^+_{r(z_0)}(x_k,z_0)$ and $B^+_{2r(z_0)}(x_k,z_0)$ respectively (observe that $z_0=1$). In green, $B'_{r(z_0)}(x_k,0)$. In blue, the ball $B^+_{r(z_1)}(x_{k,1},z_1)$ constructed in the first step. In gray, another such ball.} 
  \label{fig:figure4}
\end{figure}

\noindent With a recursion on the last step, we identify a sequence of heights $z_i := (\frac{9}{10})^i z_0 $ and $2^n$ balls at height $z_{i}$ and centers $(x_{k,1}, z_i), ..., (x_{k,2^n},z_i)$ for each of the balls $B^+_{r(z_{i-1})}(x_k,z_{i-1})$ from the previous step such that $$B^+_{r(z_i)}(x_{k,j}, z_i)\subset B^+_{2r(z_{i-1})}(x_k,z_{i-1})\quad \mbox{ for all}\quad  j\in\{1,...,2^n\}$$ and $$ B'_{r(z_{i-1})}(x_k,0)\times [z_i,z_{i-1}] \subseteq \bigcup_j B^+_{r(z_i)}(x_{k,j},z_i)\;.$$ 
We choose $N_3 := \frac{\log h}{\log (9/10)} \leq 10 |\log h|$ in order to have $z_{N_3}=h$, and see that $$\Omega\times[h,1]\subseteq \bigcup_{i,k,j} B^+_{r(z_i)}(x_{k,j}, z_i)\;,$$
where the union has at most $ 2^{n\,N_3} N_\Omega$ elements. Then by setting \begin{align*}N+N_1+N_2+N_3 &\leq 14|\log r_W| + 5|x_W-x_\Omega| + 5r_\Omega + 10|\log h| \\ & \lesssim r_\Omega\,|\log r_W|\,|x_W-x_\Omega|\,|\log h| =: C_{\Omega, W}|\log h|,\end{align*}
which holds for $h$ smaller than a numerical constant $k\in (0,1)$ and a good initial choice of $r_W$, $r_\Omega$ and $x_\Omega$, we obtain
\begin{equation*}\begin{split} \|x_{n+1}^{1/2-s}\widetilde w\|_{L^2(\Omega \times [h,1])} & \leq 2^{n\,N_3}N_\Omega \sup_{i,j,k} \|x_{n+1}^{1/2-s}\widetilde w\|_{L^2(B^+_{r(z_i)}(x_{k,j},z_i))}
\\ & \leq 2^{n\,N_3}N_\Omega K^{\frac{1 - \alpha^{C_{\Omega, W}|\log h|}}{1-\alpha}}E   (\eta/E)^{\alpha^{C_{\Omega, W}|\log h|}} \\ & \leq K^{\frac{1 }{1-\alpha}}E N_\Omega 2^{10 n |\log h|}    (\eta/E)^{\alpha^{C_{\Omega, W}|\log h|}} \leq K^{\frac{1 }{1-\alpha}} N_\Omega E\,h^{-10 n}    (\eta/E)^{h^{C_{\Omega, W}|\log \alpha|}}\;.  \end{split}\end{equation*}
For the term $\|x_{n+1}^{1/2-s}\widetilde w\|_{L^2(\Omega \times [0,h])}$ we estimate as in the proof of Theorem 5.5 in \cite{RS17} and get
\begin{equation}\label{final-estimate}\begin{split} \|x_{n+1}^{1/2-s}\widetilde w\|_{L^2(\Omega \times [0,1])} & \leq \widetilde K E \left(h^{-10 n}    (\eta/E)^{h^{C_{\Omega, W}|\log \alpha|}} + h^\beta \right)\;, \end{split}\end{equation}
where $\beta \in (0,1)$ depends on $n,s$ and $\widetilde K>0$ depends on $n,s,\Omega$. We now optimize estimate \eqref{final-estimate} by choosing $h\in(0,1)$. Assume $\eta/E \leq e^{-1}$, which ensures $z:=|\log (\eta/E)|\geq 1$, and then let $h:= (k^{-D}+z^{1/2})^{-\frac{1}{D}}$, where for simplicity we set $D:=C_{\Omega,W}|\log \alpha|$. Observe that $h<k$, as wanted. After some computations we get 
\begin{align*} h^{-10 n}    (\eta/E)^{h^{C_{\Omega, W}|\log \alpha|}} + h^\beta &= (k^{-D}+z^{1/2})^{\frac{10n}{D}}e^{-\frac{z}{k^{-D}+z^{1/2}}}+ (k^{-D}+z^{1/2})^{-\frac{\beta}{D}} \\ & \leq (k^{-D}+z^{1/2})^{\frac{10n}{D}+m} m^m (ez)^{-m} + z^{-\frac{\beta}{2D}} \\ & \leq e^{-m}m^m (2k^{-D})^{\frac{10n}{D}+m}z^{\frac{10n}{2D}-\frac{m}{2}} + z^{-\frac{\beta}{2D}} \\ & = \left(1+ e^{-m}m^m (2k^{-D})^{\frac{10n}{D}+m} \right) z^{-\frac{\beta}{2D}} \\ &  \leq (1+m^m(2k^{-D})^{2m}) z^{-\frac{\beta}{2D}} = (1+m^m2^{2m}k^{-2(10n+\beta)}) z^{-\frac{\beta}{2D}} \;.\end{align*}
Here we first used the formula $e^{-x}\leq m^m (ex)^{-m}$, which holds for all $x,m > 0$, and then chose $m:= \frac{10n+\beta}{D}$. If now $D$ is large enough, which can always be arranged by choosing $r_W$ small enough with respect to fixed $n,s,\Omega$, we get $m^m2^{2m}\leq 1$. Thus, coming back to the symbols we had in \eqref{final-estimate},
\begin{equation*}
    \|x_{n+1}^{1/2-s}\widetilde w\|_{L^2(\Omega \times [0,1])} \leq c'_{n,s}\widetilde K E |\log (\eta/E)|^{-\frac{\beta}{2C_{\Omega,W}|\log \alpha|}}
\end{equation*}
for $r_W$ small enough with respect to $n,s,\Omega$. 
\vspace{3mm}

Since the same sequence of balls can be used again, for the gradient term $\|x_{n+1}^{1/2-s}\nabla\widetilde w\|_{L^2(\Omega \times [0,1])} $ one gets a similar estimate (check Section 5.4 in \cite{RS17}):
\begin{equation*}
    \|x_{n+1}^{1/2-s}\widetilde \nabla w\|_{L^2(\Omega \times [0,1])} \leq c'_{n,s}\widetilde K' E |\log (\eta/E)|^{-\frac{\beta'}{2C_{\Omega,W}|\log \alpha'|}}\;.
\end{equation*}
Thus eventually
\begin{equation}\label{CS-estimate}
    \|x_{n+1}^{1/2-s}\widetilde w\|_{L^2(\Omega \times [0,1])} + \|x_{n+1}^{1/2-s}\widetilde \nabla w\|_{L^2(\Omega \times [0,1])} \leq c_{n,s,\Omega} E |\log(\eta/E)|^{-\frac{c_{n,s}}{C_{\Omega,W}}}\;.
\end{equation}

 Adapting Lemma 6.1 from \cite{RS17} to our case, we immediately get the estimate $$ \|v\|_{H^{-s}(\Omega)}\leq (1+\|\Psi\|_{B(H^s_{\overline\Omega}, H^{-s}(\Omega))})\|w\|_{H^s_{\overline \Omega}} \leq (1+\|\Psi\|_{B(L^2(\Omega),L^2(\Omega))})\|w\|_{H^s_{\overline \Omega}} \leq (1+M)\|w\|_{H^s_{\overline \Omega}}\;. $$
Moreover, choosing a bounded set $\Omega'$ such that $\Omega\Subset\Omega'$, by Lemma 4.4 in \cite{RS17} we get 
$$ \|w\|_{H^s_{\overline\Omega}} \leq C_{n,s,\Omega,\Omega'} \left( \|x_{n+1}^{1/2-s}\widetilde w\|_{L^2( \Omega'\times[0,1])}+ \|x_{n+1}^{1/2-s}\nabla\widetilde w\|_{L^2(\Omega'\times[0,1])} \right)\;. $$
Using these and \eqref{CS-estimate} (observe that $\Omega' \subset B_{r_\Omega}(x_\Omega)$ for a good initial choice of $r_\Omega, x_\Omega$) we get for $E$ large enough
$$\|v\|_{H^{-s}(\Omega)}\leq c_{n,s,\Omega,\Omega'}(1+M)E |\log (\eta/E)|^{-\frac{c_{n,s}}{C_{\Omega,W}}} \leq c_{n,s,\Omega,\Omega'}(1+M)E |\log (\|(-\Delta)^sw\|_{H^{-s}(W)})|^{-\frac{c_{n,s}}{C_{\Omega,W}}}\;.$$
This is almost the wanted inequality from Theorem \ref{logestimate}. In order to conclude the argument, we just need to observe that even if $W$ is not a ball, we can certainly find a ball $V\subset W$, and perform the estimate with $V$ instead than $W$. This leads to
$$\|v\|_{H^{-s}(\Omega)}\leq c_{n,s,\Omega,\Omega'}(1+M)E|\log(\|(-\Delta)^sw\|_{H^{-s}(W)})|^{-\frac{c_{n,s}}{C_{\Omega,V}}}\;.$$

We can now eventually define the constants $\sigma, c$ corresponding to $\Omega, W$. We identify a suitable small ball $V\subset W$, and let $$\sigma := \frac{c_{n,s}}{C_{\Omega,V}}\;,\quad c:= c_{n,s,\Omega,\Omega'}(1+M)E\;.$$ The dependence of $\sigma, c$ on $M,W$ is now clear. In particular, one sees that $\sigma$ presents the expected vanishing behaviour as the distance between $W$ and $\Omega$ increases.

\bibliography{References-006}

\end{document}